\documentclass[11pt]{amsart}
\usepackage{pgf,tikz,pgfplots,color}
\pgfplotsset{compat=1.15}
\usepackage{mathrsfs}
\usetikzlibrary{arrows}
\usepackage{fancyhdr}
\usepackage{lastpage} 

\usepackage[margin=0.9in]{geometry}
\usepackage{hyperref}
\hypersetup{
    colorlinks=true,
    linkcolor=blue,
    citecolor=brown,
    filecolor=magenta,
    urlcolor=blue,
}
\usepackage{amsmath}
\usepackage{amsfonts}
\usepackage{amssymb}
\usepackage{amsthm}
\usepackage{setspace}
\usepackage{inputenc}
\usepackage{comment}
\usepackage{mathtools}

\usepackage[shortlabels]{enumitem}
\usepackage{tikz-cd}

\numberwithin{equation}{section}

\newtheorem{thm}{Theorem}[section]

\newtheorem{cor}[thm]{Corollary}
\newtheorem{lem}[thm]{Lemma}
\newtheorem{prop}[thm]{Proposition}

\theoremstyle{definition}

\newtheorem{defn}[thm]{Definition}
\newtheorem{conv}[thm]{Convention}

\newtheorem*{defn*}{Definition}

\newtheorem{rmk}[thm]{Remark}

\newtheorem*{rmk*}{Remark}
\newtheorem*{ack}{Acknowledgment}

\newcommand\longmapsfrom{\mathrel{\reflectbox{\ensuremath{\longmapsto}}}}

\newcommand{\1}{\mathbf{1}}

\newcommand{\Bs}{\mathscr{B}}

\newcommand{\N}{\mathbb{N}}

\newcommand{\Pc}{\mathcal{P}}

\newcommand{\R}{\mathbb{R}}

\newcommand{\Ss}{\mathscr{S}}

\newcommand{\Z}{\mathbb{Z}}

\newcommand{\eps}{\varepsilon}

\newcommand{\Span}{\operatorname{span}}

\newcommand{\supp}{\operatorname{supp}}

\title[Non-compact Besov Embeddings]{A Classification Theorem on Non-compact Embeddings between Besov Spaces}

\author[Chian Yeong Chuah]{Chian Yeong Chuah}
\address{Department of Mathematics, University of Georgia, Athens, GA 30602, United States}
\email{chianyeong.chuah@uga.edu}

\author[Jan Lang]{Jan Lang}
\address{\parbox{\linewidth}{Department of Mathematics, The Ohio State University, Columbus, OH 43210, United States\\
Department of Mathematics, Faculty of Electrical Engineering, Czech Technical University in Prague, 166 27 Prague, Czech Republic\vspace{0.055in}}}
\email{lang@math.osu.edu}

\author[Liding Yao]{Liding Yao} 
\address{Department of Mathematical Sciences,
	Purdue University Fort Wayne, Fort Wayne, IN 46805, United States} 
\email{yao388@purdue.edu}

\begin{document}

\begin{abstract}
    We analyze the embedding properties between Besov spaces, defined on the total space $\mathbb R^n$ and on bounded domains. We give a complete classification on whether or not these embedding maps satisfy certain weaker forms of compactness characterized by the so-called strictly and finitely strictly singular conditions. The result extends the recent findings on Sobolev embeddings by offering a refined description of the quality of non-compactness in the setting of Besov spaces.
\end{abstract}

\thanks{L. Yao is supported in part by AMS-Simons Travel Grant}

\maketitle

\section{Introduction}

It is well known that Sobolev embeddings on domains (say a bounded Lipschitz domain $\Omega\subset\R^n$) into the optimal Lebesgue space and the optimal Lorentz space
\begin{align*}\label{Sobolev 1}
	& W^{k,p}(\Omega) \hookrightarrow L^\frac{np}{n-kp}(\Omega) \quad  \text{and} \quad W^{k,p}(\Omega) \hookrightarrow  L^{\frac{np}{n-kp},p}(\Omega),
\end{align*}
where \(1 \leq k < n\) and \(1<p<n/k\) are non-compact. 
Given that \(L^{\frac{np}{n-kp},p}(\Omega) \subsetneq L^{\frac{np}{n-kp}}(\Omega)\), one might expect a significant difference in the ``quality'' of non-compactness for these embeddings. However, standard measures of non-compactness, such as the decay of entropy numbers or Kolmogorov numbers, fail to differentiate between these embeddings, as both are maximally non-compact (i.e., the measure of non-compactness coincides with the norm of the embeddings, see \cite{LMOP}).

In \cite{LangMihula}, it has been observed that the first Sobolev embedding is finitely strictly singular (i.e., the Bernstein numbers decay to 0), while the second Sobolev embedding is not strictly singular, with all Bernstein numbers equal to the norm of the embedding. This observation suggests that strict singularity and finitely strict singularity could provide a more refined distinction between non-compact embeddings, offering a better description of the ``quality'' of non-compactness comparable with the measure of non-compactness.

The above observations regarding Sobolev embeddings naturally lead to the question of studying the ``quality'' of non-compactness for embeddings between other function spaces.

In this paper, we consider the embeddings between Besov spaces \[ \Bs_{p_0q_0}^{s_0}\subset \Bs_{p_1q_1}^{s_1}, \mbox{ both on bounded domains and on } \R^n. \]
We provide a comprehensive and complete classification of the ``quality'' of non-compactness for the embeddings between Besov spaces in the context of strict singularity and finitely strict singularity, as demonstrated in the following two theorems which are the main results of this paper.

For embeddings between Besov spaces on bounded domains, we obtain:
\begin{thm}\label{Thm::ClassifyDom}
    Let $p_0,p_1,q_0,q_1\in(0,\infty]$, $s_0,s_1\in\R$. Let $\Omega\subset\R^n$ be a bounded Lipschitz open subset. Consider the embedding $\Bs_{p_0q_0}^{s_0}(\Omega)\hookrightarrow \Bs_{p_1q_1}^{s_1}(\Omega)$.
    \begin{enumerate}[(i)]
        \item\label{Item::ClassifyDom::Not} $\Bs_{p_0q_0}^{s_0}(\Omega)\not\subset \Bs_{p_1q_1}^{s_1}(\Omega)$, i.e. there is no embedding, if and only if:
        \begin{itemize}
            \item $s_0-s_1<\max(0,\frac n{p_0}-\frac n{p_1})$; or
            \item $s_0-s_1=\max(0,\frac n{p_0}-\frac n{p_1})$ and $q_0>q_1$.
        \end{itemize}
        \item\label{Item::ClassifyDom::Cpt} The embedding
         is compact if and only if $s_0-s_1>\max(0,\frac n{p_0}-\frac n{p_1})$.
        \item\label{Item::ClassifyDom::FSS} The embedding is non-compact but finitely strictly singular if and only if:
        \begin{itemize}
            \item $s_0-s_1=\frac n{p_0}-\frac n{p_1}>0$ and $q_0<q_1$; or
            \item $s_0-s_1=0$, $p_1<p_0=\infty$ and $q_0<q_1$.
        \end{itemize}
        \item\label{Item::ClassifyDom::SS} The embedding is not finitely strictly singular but strictly singular if and only if:
        \begin{itemize}
            \item $s_0-s_1=0$, $p_1=p_0=\infty$ and $q_0<q_1$; or
            \item $s_0-s_1=0$, $p_1\le p_0<\infty$ and $q_0<q_1$.
        \end{itemize}
        \item\label{Item::ClassifyDom::NSS} The embedding is not strictly singular if and only if $s_0-s_1=\max(0,\frac n{p_0}-\frac n{p_1})$ and $q_0=q_1$.
\end{enumerate}
\end{thm}
\begin{rmk*}
    By applying the standard partition of unity argument, the same results in Theorem~\ref{Thm::ClassifyDom} hold if we replace $\Omega$ by a compact smooth manifold. We refer to \cite[Chapter~7]{TriebelTheoryOfFunctionSpacesII} for the definition of spaces and leave the details to readers.
\end{rmk*}

For embeddings between Besov spaces on $\R^n$, we obtain:
\begin{thm}\label{Thm::ClassifyRn}
    Let $p_0,p_1,q_0,q_1\in(0,\infty]$, $s_0,s_1\in\R$. Consider the embedding $\Bs_{p_0q_0}^{s_0}(\R^n)\hookrightarrow \Bs_{p_1q_1}^{s_1}(\R^n)$.
    \begin{enumerate}[(i)]
        \item\label{Item::ClassifyRn::Not} $\Bs_{p_0q_0}^{s_0}(\R^n)\not\subset \Bs_{p_1q_1}^{s_1}(\R^n)$, i.e. there is no embedding, if and only if one of the following occurs:
        \begin{itemize}
            \item $p_0>p_1$;
            \item $s_0-s_1<\frac n{p_0}-\frac n{p_1}$;
            \item $s_0-s_1=\frac n{p_0}-\frac n{p_1}\ge0$ and $q_0>q_1$.
        \end{itemize}
        \item\label{Item::ClassifyRn::FSS} The embedding is non-compact but finitely strictly singular if and only if one of the following:
        \begin{itemize}
            \item $s_0-s_1=\frac n{p_0}-\frac n{p_1}>0$ and $q_0<q_1$;
            \item $s_0-s_1>\frac n{p_0}-\frac n{p_1}>0$.
        \end{itemize}
        \item\label{Item::ClassifyRn::NSS} The embedding is not strictly singular if and only if one of the following occurs:
        \begin{itemize}
            \item $s_0-s_1>0$ and $p_0=p_1$;
            \item $s_0-s_1=0$, $p_0=p_1$ and $q_0\le q_1$;
            \item $s_0-s_1=\frac n{p_0}-\frac n{p_1}>0$ and $q_0=q_1$.
        \end{itemize}
\end{enumerate}
\end{thm}
We also include a result of embedding of homogeneous Besov spaces $\dot \Bs_{p_0q_0}^{s_0}(\R^n)\hookrightarrow \dot \Bs_{p_1q_1}^{s_1}(\R^n)$, see Theorem~\ref{Thm::ClassifyHomo}.

In Theorems~\ref{Thm::ClassifyDom} and \ref{Thm::ClassifyRn}, the case where the map is unbounded and the case where it is compact are both well-known. The critical exponents cases, which form the core of our paper, fall in the  following two cases:
\begin{itemize}
    \item The analog of Sobolev embedding: $s_0-s_1=\frac n{p_0}-\frac n{p_1}>0$ and $q_0<q_1$. We  show that both $\Bs_{p_0q_0}^{s_0}(\R^n)\hookrightarrow \Bs_{p_1q_1}^{s_1}(\R^n)$ and $\Bs_{p_0q_0}^{s_0}(\Omega)\hookrightarrow \Bs_{p_1q_1}^{s_1}(\Omega)$ are finitely strictly singular. See Proposition~\ref{Prop::SobEmbed}.
    \item The analog of $L^p$-$L^q$ embedding on domains: $s_0=s_1$, $p_0\ge p_1$ and $q_0<q_1$. In Proposition~\ref{Prop::DomEmbed} we show that for the map  $\Bs_{p_0q_0}^{s_0}(\Omega)\hookrightarrow \Bs_{p_1q_1}^{s_1}(\Omega)$:
    \begin{itemize}
        \item When $p_0=\infty> p_1$, it is finitely strictly singular.
        \item When $p_0=p_1=\infty$ or $p_0<\infty$, it is strictly singular but not finitely strictly singular. 
    \end{itemize}
\end{itemize}

We would like to point out that the above results on the ``quality'' of non-compactness for embeddings between Besov spaces closely align with the findings on strict singularity for embeddings between Lorentz sequence spaces \(l^{p,q}\) (see \cite{LangNekvinda}). The comparison is both interesting and noteworthy.

It is also important to highlight that the Sobolev embedding $W^{k,p}(\Omega)$ into the space \(L^{\frac{np}{n-kp},p}(\Omega)\) is not only the optimal (i.e., smallest) target space within the scale of Lorentz spaces but also the optimal target space among all rearrangement-invariant spaces (such as Lebesgue spaces, Lorentz spaces, Orlicz spaces, and similar spaces). Specifically, this implies that if \(Z(\Omega)\) is a rearrangement-invariant space such that \(W_0^{k,p}(\Omega) \subset Z(\Omega)\), then \(L^{\frac{np}{n-kp},p}(\Omega) \subseteq Z(\Omega)\) (see \cite[Theorem A]{KerPick}). In contrast, the Sobolev embedding into \(L^\frac{np}{n-kp}(\Omega)\) is optimal within the narrower scale of Lebesgue spaces, and the ``quality'' of non-compactness differs significantly (i.e., non-strictly singular versus finitely strictly singular). This observation suggests that non-strictly singular embeddings may be associated with the ``most optimal'' embeddings. This insight could help identify cases where we might expect to find the ``most optimal'' embeddings between Besov spaces, as seen in the case \ref{Item::ClassifyDom::NSS} in Theorem \ref{Thm::ClassifyDom} and case \ref{Item::ClassifyRn::NSS} in Theorem \ref{Thm::ClassifyRn}, comparable with cases in which there is a chance for 
more optimal target spaces, i.e. when embedding is finitely strictly singular or strictly singular.

It is also worth noting that the techniques used in \cite{LangMihula}, \cite{LangMusil} and \cite{BouGro} for Sobolev embeddings are fundamentally different from those employed in this paper. In particular, the results for Sobolev embeddings \(W^{k,p}\to L^\frac{np}{n-kp}\) require quite sophisticated methods. This contrasts with the techniques used for Besov embeddings and highlights the advantages of working with Besov spaces compared to Sobolev spaces and Lorentz spaces. 
This is demonstrated in our forthcoming paper \cite{ChuahLangYaoSobolev}, in which, by using the results obtained for Besov spaces and their relation with Sobolev spaces, we improve and extend findings from \cite{BouGro} and \cite{LangMihula}.

Additionally, it is noteworthy that, to the best of our knowledge, there is no natural Sobolev embedding that is non-compact, strictly singular, yet not finitely strictly singular. This specific scenario, however, can be naturally observed within the framework of Besov spaces, as demonstrated in Theorem \ref{Thm::ClassifyDom} and Theorem \ref{Thm::ClassifyRn}, underscoring the structural richness of Besov spaces.

The paper is structured as follows. In the next section, we recall essential definitions and notations used throughout the paper, along with results concerning the wavelet decomposition of Besov spaces and their implications. Section~\ref{Section::SeqSpace} focuses on the study of embeddings between sequence spaces. In Section~\ref{Section::PfThm}, we present the proofs of the main theorems, Theorems~\ref{Thm::ClassifyDom} and~\ref{Thm::ClassifyRn}. Section~\ref{Section::HomoEmbed} extends Theorem~\ref{Thm::ClassifyRn} to the setting of homogeneous Besov spaces, resulting in Theorem~\ref{Thm::ClassifyHomo}. 
Finally, Section~\ref{Section::Further} concludes the paper with final remarks and observations.

\section{Preliminaries}
In this section, we recall some notations, definitions, and auxiliary results.

Throughout the paper, by linear spaces, we mean the real linear spaces. We focus on quasi-Banach spaces as most of the results we used on Banach spaces remain the same for quasi-Banach spaces.

 By $\xrightarrow{\simeq}$ we denote bounded linear isomorphism between (quasi-)Banach spaces and $\Subset$ will stand for compact inclusion of bounded sets.
\begin{defn}
    Let $T:X\to Y$ be a bounded linear map between two (real) quasi-normed spaces $X$ and $Y$.

    We say $T$ is \textit{strictly singular} if there is no infinite-dimensional subspace $Z$ of $X$ such that the restriction $T|_Z$ is a topological isomorphism of $Z$ onto $T(Z)$. Equivalently, we can say that
\[
\inf_{x\in Z;\|x\|_X=1}\|T(x)\|_Y= 0,\quad\text{for every subspace }Z\subseteq X\text{ such that }\dim Z=\infty.
\]

    We say $T$ is \textit{finitely strictly singular} if for any $\varepsilon > 0$, there exists $N(\varepsilon) \in \Z_+$ such that if $E$ is a subspace of $X$ with $\dim E \ge N(\varepsilon)$, then there exists $x \in E$ with $\|x\|_X=1$, such that $\|Tx\|_Y \le \varepsilon$.
\end{defn}

The finite strict singularity can also be expressed in terms of the Bernstein numbers of $T$. Observe that $T$ is finitely strictly singular if and only if $\lim_{n\to\infty}b_n(T)=0$.

\begin{defn}
    For $n\ge1$, the $n$-th \textit{Bernstein number} $b_n(T)$ of a bounded linear map $T:X\to Y$ is
\[
b_{n} (T) := \sup \big\{ \inf_{x \in Z_n; \| x \|_{X} = 1} \left\| T(x) \right\|_{Y} : Z_n\subseteq X\text{ is a subspace such that }\dim Z_n\ge n\big\}.
\]
\end{defn}

Obviously, we have:
\[ T \text{ is compact} \implies T \text{ is finitely strictly singular} \implies T \text{ is strictly singular.} \]

To classify embedding between Besov spaces, we reduce the discussion to sequence spaces.

\begin{conv}
    Let $q\in(0,\infty]$ and $S$ be a set, we denote by $\ell^q(S)$ the spaces of $q$-summable sequences $x=(x_j)_{j\in S}$ with quasinorm $\|x\|_{\ell^q(S)}:=(\sum_{j\in S}|x_j|^q)^{1/q}$ when $q<\infty$ and $\|x\|_{\ell^\infty(S)}=\sup_{j\in S}|x_j|$. 
        When $S=\N_0=\{0,1,2,\dots\}$ we use abbreviation $\ell^q=\ell^q(\N_0)$.
\end{conv}
\begin{defn}\label{Defn::VectEllSpace}
    Let $(X_j)_{j=0}^\infty$ be a family of quasi-normed spaces and let $q\in(0,\infty]$. We denote by $\ell^q(X_j)_{j=0}^\infty=\ell^q(\N_0;(X_j)_{j=0}^\infty)$ the space of vector-valued sequences $x=(x_0,x_1,x_2,\dots)$ where $x_j\in X_j$ for each $j\ge0$, such that
    \begin{equation*}
        \|x\|_{\ell^q(X_j)_{j=0}^\infty}:=\Big(\sum_{j=0}^\infty \|x_j\|_{X_j}^q\Big)^{1/q}.
    \end{equation*}
    Here when $q=\infty$ we replace the sum by supremum over $j\ge0$.

    We also write $x=\sum_{j=0}^\infty x_j\otimes e_j$ where $e_j=(0,\dots,0,\underset j1,0,\dots)$ is the standard basis in $\ell^q$.

    When $(X_j)_{j=0}^\infty$ are all identical for all $j\ge0$, say $X_j\equiv Y$. We use the notation $\ell^q(Y)=\ell^q(X_j)_{j=0}^\infty$.
\end{defn}
\begin{rmk}\label{Rmk::VectEllSpace}
    Using the notation in the definition, for the double sequence space $\ell^q(\ell^p)$, by the element $e_j\otimes e_k$ we mean $e_j\in\ell^p$ and $e_k\in\ell^q$. More precisely $e_j\otimes e_k=(0,\dots,0,\underset k{e_j},0,\dots)$.
\end{rmk}
\begin{rmk}
    Notice that $\ell^q(X_j)_{j=0}^\infty$ may not be quasi-normed since $(X_j)_{j=0}^\infty$ may not have uniform triangle inequality. That is, one may not find a $C>0$ such that $\|x_j'+x_j''\|_{X_j}\le C(\|x_j'\|_{X_j}+\|x''_j\|_{X_j})$ for all $j\ge0$ and $x_j',x_j''\in X_j$. For example, we can take $X_j=\ell^{1/j}$ for $j\ge1$. However,
        in our case $(X_j)_{j=0}^\infty$ are always subspaces of $\ell^p$ for some $p>0$, where we have $\|x_j'+x_j''\|_{X_j}\le 2^{\max(1/p-1,0)}(\|x_j'\|_{X_j}+\|x''_j\|_{X_j})$ for all $j$ and $x_j',x_j''$.
\end{rmk}

\begin{conv}\label{Conv::LambdaX}
    Let $(X,\|\cdot\|_X)$ be a quasi-normed space and let $\lambda>0$. We denote by $\lambda X$ the quasi-normed space $(X,\|\cdot\|_{\lambda X})$ where $\|x\|_{\lambda X}:=\lambda\|x\|_X$ for every $x\in X$.
\end{conv}
\begin{rmk}\label{Rmk::LambdaX}
    For every finite set $S$ and $0<r\le\infty$, the space $(\# S)^{-1/r}\cdot\ell^r(S)$ equals $L^r(S,\mu_S)$ where $\mu_S\{a\}=(\# S)^{-1}$ for every $a\in S$. 
    
    Since $(S,\mu_S)$ is a probability space, for every $0<r_1\le r_0\le\infty$, the embedding $$L^{r_0}(S,\mu_S)=(\# S)^{-1/r_0}\cdot\ell^{r_0}(S)\hookrightarrow L^{r_1}(S,\mu_S)=(\# S)^{-1/r_1}\cdot\ell^{r_1}(S)$$ has operator norm 1.
\end{rmk}

\begin{defn}\label{Defn::BesovRn}
    Let $s\in \R$ and $p,q \in (0,\infty]$. We denote by $\Bs_{pq}^s(\R^n)$ the Besov space which contains all tempered distributions $f\in \Ss'(\R^n)$ for which
\[
\|f\|_{\Bs_{pq}^s(\R^n)} = \|f\|_{\Bs_{pq}^s(\R^n;\varphi)} := \Big( \sum_{j=0}^{\infty} \| 2^{js}\varphi_j\ast f\|_{L^p(\R^n)}^q \Big)^{\frac{1}{q}} <\infty.
\]
    Here $(\varphi_j)_{j=0}^\infty\subset\Ss(\R^n)$ is a fixed family of Schwartz functions whose Fourier transform $\hat\varphi_j(\xi)=\int \varphi_j(x)e^{-2\pi ix\xi}dx$ satisfy
\begin{itemize}
    \item $\supp\hat\varphi_0\subset B(0,2)$ and $\hat\varphi_0\equiv1$ in a neighborhood of $\overline {B(0,1)}$.
    \item $\hat\varphi_j(\xi)=\hat\varphi_0(2^{-j}\xi)-\hat\varphi_0(2^{1-j}\xi)$ for $j\ge1$. 
\end{itemize}
\end{defn}

Different choices of $(\varphi_j)_j$ result in different but equivalent norms (see e.g. \cite[Proposition~2.3.2]{TriebelTheoryOfFunctionSpacesI}), thus, the space is independent of $(\varphi_j)_j$ even though its norm is dependent on the choice of functions. In this paper, we do not use the exact norm definition from above but the discretized version via wavelet, see Proposition~\ref{Prop::Wavelet}.

\begin{defn}[{See also \cite[Definition~1.95]{TriebelTheoryOfFunctionSpacesIII}}]\label{Defn::BesovDom}
For an open subset $\Omega\subseteq\R^n$, we define $\Bs_{pq}^s(\Omega)$ by the restriction of $\Bs_{pq}^s(\R^n)$ functions to $\Omega$:
\begin{equation*}
    \Bs_{pq}^s(\Omega):= \{\tilde f|_\Omega:\tilde f\in \Bs_{pq}^s(\R^n) \},\quad\|f\|_{\Bs_{pq}^s(\Omega)}:=\inf_{\tilde f|_\Omega=f}\|\tilde f\|_{\Bs_{pq}^s(\R^n)}.
\end{equation*}

We define the closed subspace $\widetilde \Bs_{pq}^s(\overline\Omega)\subseteq \Bs_{pq}^s(\R^n)$ by
\begin{equation*}
    \widetilde \Bs_{pq}^s(\overline\Omega):= \{\tilde f\in \Bs_{pq}^s(\R^n):\supp \tilde f\subseteq\overline\Omega \}. 
\end{equation*}
\end{defn}

\begin{rmk}
    From the above definition, we have the following quotient characterization: for $\Omega\subset\R^n$,
\begin{equation}\label{Eqn::QuoChar}
    \Bs_{pq}^s(\Omega)=\Bs_{pq}^s(\R^n)/\{\tilde f\in \Bs_{pq}^s(\R^n):\tilde f|_\Omega=0\}.
\end{equation}

When $\Omega=\R^n$ we have by definition the coincidence of notations
\begin{equation*}
    \Bs_{pq}^s(\R^n)=\widetilde \Bs_{pq}^s(\overline{\R^n}).
\end{equation*}
\end{rmk}

\begin{lem}\label{Lem::TrivialInclusion}
    Let $\Omega_0,\Omega_1\subseteq\R^n$ be two open subsets such that $\Omega_0$ is precompact in $\Omega_1$. We have the natural topological embedding $\widetilde \Bs_{pq}^s(\overline\Omega_0)\hookrightarrow \Bs_{pq}^s(\Omega_1)$.
\end{lem}

\begin{proof}
    An $f\in \widetilde \Bs_{pq}^s(\overline\Omega_0)$ is an element in $\Bs_{pq}^s(\R^n)$, thus by definition and \eqref{Eqn::QuoChar} $\|f\|_{\Bs_{pq}^s(\Omega_1)}\le\|f\|_{\Bs_{pq}^s(\R^n)}$.

    On the other hand, take a $\chi\in C_c^\infty(\Omega_1)$ such that $\chi\equiv1$ in a neighborbood of $\overline\Omega_0$. Recall from e.g. \cite[Theorem~2.8.2]{TriebelTheoryOfFunctionSpacesI} that $[\tilde f\mapsto \chi\tilde f]:\Bs_{pq}^s(\R^n)\to \Bs_{pq}^s(\R^n)$ is a bounded map whose kernel contains $\{\tilde f:\tilde f|_{\Omega_1}=0\}$. By the quotient characterization \eqref{Eqn::QuoChar} it descends to a bounded linear map $[f\mapsto \chi f]:\Bs_{pq}^s(\Omega_1)\to \Bs_{pq}^s(\R^n)$, i.e. there is a $C>0$ such that $\|\chi f\|_{\Bs_{pq}^s(\R^n)}\le C\|f\|_{\Bs_{pq}^s(\Omega_1)}$ for all $f\in \Bs_{pq}^s(\Omega_1)$.
    
    Since $f=\chi f$ for every $f\in \widetilde \Bs_{pq}^s(\overline\Omega_0)$, this is to say $\|f\|_{\Bs_{pq}^s(\R^n)}=\|\chi f\|_{\Bs_{pq}^s(\R^n)}\le C\|f\|_{\Bs_{pq}^s(\Omega_1)}$ for all $f\in \widetilde \Bs_{pq}^s(\overline\Omega_0)$, completing the proof.
\end{proof}

The mapping between Besov spaces can be reduced to the corresponding sequence spaces via wavelet decomposition theorems. The following proposition is quoted from \cite[Theorem~1.20]{TriebelSpacesOnDomains}, where a wavelet basis of Daubechies type is used. Earlier spline-basis representations of Besov spaces on cubes and compact smooth manifolds were obtained by Ciesielski and Figiel \cite{CF1,CF2}.

In the following we use indices sets $Q_j$ and $G_j$ for $j\ge0$, given by
\begin{equation}
    Q_j:=2^{-j}\Z^n\text{ for }j\ge0;\qquad G_0:=\{0,1\}^n,\quad G_j:=\{0,1\}^n\backslash\{0\}^n\text{ for }j\ge1.
\end{equation}
\begin{prop}[{\cite[Theorem~1.20]{TriebelSpacesOnDomains}}]\label{Prop::Wavelet}
    Let $\eps>0$. There are a $N=N(\eps)\in\N$ and a set $\{\psi_{jm}^g:j\ge0,m\in Q_j,g\in G_j\}\subset C_c^N(\R^n)$ that forms a (real) orthonormal basis of $L^2(\R^n)$ satisfying the following.
    \begin{enumerate}[(i)]
        \item\label{Item::Wavelet::Scal} $\psi_{jm}^g(x)=2^{jn/2}\psi_{00}^g(2^j(x-m))$ for all $j\ge0$, $m\in Q_j$ and $g\in G_j$.
        \item\label{Item::Wavelet::Supp} There is a $L\ge1$ such that $\supp\psi_{00}^g\subset B(0,2^L)$ for all $g\in G_0$. Thus $\supp\psi_{jm}^g\subset B(m,2^{L-j})$.
        \item\label{Item::Wavelet::Sum} For every $p,q\in[\eps,\infty],s\in[-\eps^{-1},\eps^{-1}]$ and $f\in \Bs_{pq}^s(\R^n)$, we have expansion (which is orthonormal in $L^2$) $$f=\sum_{j=0}^\infty\sum_{m\in Q_j,g\in G_j}(f,\psi_{jm}^g)\cdot\psi_{jm}^g,$$ where the series converges in the sense of tempered distributions.    
        \item\label{Item::Wavelet::Isom} For every $p,q\in[\eps,\infty],s\in[-\eps^{-1},\eps^{-1}]$, 
        \begin{equation}\label{Eqn::Wavelet::bpqsSpace}
            f\in \Bs_{pq}^s(\R^n)\quad\text{ if and only if }\quad\bigg(\sum_{j=0}^\infty 2^{j(s-\frac np+\frac n2)q}\Big|\sum_{m\in Q_j,g\in G_j}|(f,\psi_{jm}^g)|^p\Big|^{q/p}\bigg)^{1/q}<\infty.
        \end{equation}
        In other words we obtain the following isomorphisms
        \begin{equation}\label{Eqn::Wavelet::Lambda}
            \begin{array}{rcl}
            \Lambda=\Lambda_\eps:\Bs_{pq}^s(\R^n)&\xrightarrow{\simeq}&\ell^q\big(2^{j(s-\frac np+\frac n2)}\cdot\ell^p(Q_j\times G_j)\big)_{j=0}^\infty;\\
            f&\longmapsto&\{(f,\psi_{jm}^g)_{L^2}:m\in Q_j,\ g\in G_j\}_{j=0}^\infty;\\
            \sum_{j=0}^\infty\sum_{m,g}a_{jm}^g\psi_{jm}^g&\longmapsfrom&\{a_{jm}^g:m\in Q_j,\ g\in G_j\}_{j=0}^\infty.
        \end{array}
        \end{equation}

        In particular, by fixing identifications $Q_j\times G_j\simeq\Z$ for each $j\ge0$ we obtain the isomorphisms $\Lambda:\Bs_{pq}^s(\R^n)\xrightarrow{\simeq}\ell^q(2^{j(s-\frac np+\frac n2)}\cdot\ell^p(\Z))_{j=0}^\infty$. 
    \end{enumerate}
\end{prop}
\begin{rmk}
    The right hand side of \eqref{Eqn::Wavelet::bpqsSpace} is actually the so-called Besov sequence norm. See \cite[Definition~1.18]{TriebelSpacesOnDomains} for details.
\end{rmk}

For more information about construction and structure of $\psi_{jm}^g$ and $G_j$ we direct readers to \cite[Section~1.2.1]{TriebelSpacesOnDomains}.

In the next we sometimes restrict the $\Lambda$ or $\Lambda^{-1}$ from \eqref{Eqn::Wavelet::Lambda} to certain subspaces. 
Here we use the existence of the extension operator on Lipschitz domains.
\begin{lem}[{\cite[Theorem~2.2]{RychkovExtension}}]\label{Lem::Extension}
    Let $\Omega\subset\R^n$ be a bounded Lipschitz domain and let $U\supset\overline\Omega$ be an open neighborhood. For every $\eps>0$, there is an extension operator $E$ such that $E:\Bs_{pq}^{s}(\Omega)\to\widetilde \Bs_{pq}^s(\overline U)$ is defined and bounded for all $p,q\in[\eps,\infty]$ and $s\in[-\eps^{-1},\eps^{-1}]$.
\end{lem}
In fact, we can construct a so-called \textit{universal extension operator} that does not depend on $\eps$. That is, we can find a single $E$ such that $E:\Bs_{pq}^{s}(\Omega)\to\widetilde \Bs_{pq}^s(\overline U)$ for all $p,q\in(0,\infty]$ and $s\in\R$. See \cite[Theorem~4.1]{RychkovExtension}. 

Here, we also remark that $E:\Bs_{pq}^s(\Omega)\to \Bs_{pq}^s(\R^n)$ is always a topological embedding since it is the right inverse to the restriction mapping $[\tilde f\mapsto\tilde f|_\Omega]:\Bs_{pq}^s(\R^n)\twoheadrightarrow \Bs_{pq}^s(\Omega)$.

\begin{prop}\label{Prop::SubspaceWavelet}
    Let $\Omega\subset\R^n$ be a bounded Lipschitz domain, and $U$ and $V$ be bounded open sets such that $V\Subset\Omega\Subset U$.    
    Let $\eps>0$, $\Lambda$ be as in \eqref{Eqn::Wavelet::Lambda} (depending on $\eps$), and let $E$ be the extension operator as in Lemma~\ref{Lem::Extension} (depending on $\Omega,U$).
    
    Then there are an integer $A\ge1$ and sets $R_j\subset S_j\subset Q_j\times G_j$ for $j\ge 0$, such that, for every $p,q\in[\eps,\infty]$ and $s\in[-\eps^{-1},\eps^{-1}]$:
    \begin{enumerate}[(i)]
        \item\label{Item::SubspaceWavelet::Small} $\# R_j=2^{(j - A) n}$ for $j\ge A$ and we have $\Lambda^{-1}:\ell^q(2^{j(s-\frac np+\frac n2)}\cdot\ell^p(R_j))_{j=A}^\infty\to\widetilde \Bs_{pq}^s(\overline V)$. As a result $\Lambda^{-1}:\ell^q(2^{j(s-\frac np+\frac n2)}\cdot\ell^p(R_j))_{j=A}^\infty\to \Bs_{pq}^s(\Omega)$ is a topological embedding.
        \item\label{Item::SubspaceWavelet::Big} $\# S_j=2^{(j+A)n}$ for $j\ge 0$ and we have $\Lambda:\widetilde \Bs_{pq}^s(\overline U)\to\ell^q(2^{j(s-\frac np+\frac n2)}\cdot\ell^p(S_j))_{j=0}^\infty$. As a result $\Lambda\circ E:\Bs_{pq}^s(\Omega)\to\ell^q(2^{j(s-\frac np+\frac n2)}\cdot\ell^p(S_j))_{j=0}^\infty$ is a topological embedding.
    \end{enumerate}
    
\end{prop}
Here, the constant $A$ depends on the number $L$ in Proposition~\ref{Prop::Wavelet}~\ref{Item::Wavelet::Supp}, the inner radius of $V$ and the diameter of $U$.
\begin{rmk}
    It is possible to show that for a bounded Lipschitz domain $\Omega$, there is a linear map $\tilde\Lambda$ such that $\tilde\Lambda:\Bs_{pq}^s(\Omega)\xrightarrow{\simeq}\ell^q(2^{j(s-\frac np+\frac n2)}\cdot\ell^p\{1,\dots,2^{nj}\})_{j=0}^\infty$ is an isomorphism for all $p,q\in[\eps,\infty]$ and $s\in[-\eps^{-1},\eps^{-1}]$. However its proof may be complicated as it requires some  re-ordering of certain subsets of $S_j$ (which contains $R_j$).
    
    But for our needs, the weaker results \ref{Item::SubspaceWavelet::Small} and \ref{Item::SubspaceWavelet::Big} are satisfactory.
\end{rmk}

\begin{proof}[Proof of Proposition~\ref{Prop::SubspaceWavelet}]
    Let $L\ge1$ be from Proposition~\ref{Prop::Wavelet}~\ref{Item::Wavelet::Supp}. We define $\mathfrak R_j,\mathfrak S_j\subset Q_j$ by $\mathfrak R_j:=\{m\in Q_j:B(m,2^{L-j})\subset V\}$ and $\mathfrak S_j:=\{m\in Q_j:B(m,2^{L-j})\cap U\neq\varnothing\}$. 
    
    Recall $Q_j=2^{-j}\Z^n$, thus for a large enough $A\ge 1$ (depending on $U,V,L$) we have $\#\mathfrak R_j\ge 2^{(j-A)n}$ (for $j\ge A$) and $\#\mathfrak S_j\le 2^{(j+A-1)n}$ (for $j\ge0$). 
    
    By Proposition~\ref{Prop::Wavelet}~\ref{Item::Wavelet::Supp} we have $\supp\psi_{jm}^g\subset V$ for all $m\in \mathfrak R_j$ and $g\in G_j$, thus by \eqref{Eqn::Wavelet::bpqsSpace} $\sum_{j=A}^\infty\sum_{m\in \mathfrak R_j,g\in G_j} a_{jm}^g\psi_{jm}^g\in\widetilde \Bs_{pq}^s(\overline V)$ holds for all $\{a_{jm}^g\}_{j=A}^\infty\in \ell^q(2^{j(s-\frac np+\frac n2)}\cdot\ell^p(\mathfrak R_j\times G_j))_{j=0}^\infty$. Taking $R_j\subseteq \mathfrak R_j\times G_j$ be any set such that $\# R_j=2^{(j-A)n}$ and applying Proposition~\ref{Prop::Wavelet}~\ref{Item::Wavelet::Isom} we get $\Lambda^{-1}:\ell^q(2^{j(s-\frac np+\frac n2)}\cdot\ell^p(R_j))_{j=A}^\infty\to\widetilde \Bs_{pq}^s(\overline V)$. Applying Lemma~\ref{Lem::TrivialInclusion} we have topological embedding $\widetilde \Bs_{pq}^s(\overline V)\hookrightarrow \Bs_{pq}^s(\Omega)$. This completes the proof of \ref{Item::SubspaceWavelet::Small}.

    Similarly by Proposition~\ref{Prop::Wavelet}~\ref{Item::Wavelet::Supp} and \ref{Item::Wavelet::Sum}, for every $f\in \widetilde \Bs_{pq}^s(\overline U)$, we have $(f,\psi_{jm}^g)_{L^2}=0$ if $m\notin \mathfrak S_j$, hence by Proposition~\ref{Prop::Wavelet}~\ref{Item::Wavelet::Sum} $f=\sum_{j=0}^\infty\sum_{m\in \mathfrak S_j,g\in G_j} (f,\psi_{jm}^g)_{L^2}\cdot\psi_{jm}^g$. We take $S_j\supseteq \mathfrak S_j\times G_j$ to be any set in $Q_j\times G_j$ such that $\# S_j=2^{(j+A)n}$, which gives $R_j\subset S_j$ since $R_j\subset  \mathfrak R_j\times G_j\subset  \mathfrak S_j\times G_j\subset S_j$. Applying Proposition~\ref{Prop::Wavelet}~\ref{Item::Wavelet::Isom} we get $\Lambda:\widetilde \Bs_{pq}^s(\overline U)\to\ell^q(2^{j(s-\frac np+\frac n2)}\cdot\ell^p(S_j))_{j=0}^\infty$. 
    
    Applying Lemma~\ref{Lem::Extension} which gives $E:\Bs_{pq}^s(\Omega)\to \widetilde\Bs_{pq}^s(\overline U)$, we get the required boundedness of $\Lambda\circ E$. Since $\Lambda$ and $E$ are both topological embeddings, so is $\Lambda\circ E$. This completes the proof of \ref{Item::SubspaceWavelet::Big}.
\end{proof}

Now we reduce the discussion from function spaces to sequence spaces.

The following well-known result can  also be deduced from the wavelet decomposition.
\begin{cor}\label{Cor::EmbedAllIndex}
    Let $s_0,s_1\in\R$ and $0<p_0,p_1,q_0,q_1\le\infty$. 
    \begin{enumerate}[(i)]
        \item\label{Item::EmbedAllIndex::Sob} The inclusion $\Bs_{p_0q_0}^{s_0}(\R^n)\subseteq \Bs_{p_1q_1}^{s_1}(\R^n)$ holds if and only if
        \begin{gather}\label{Eqn::EmbedAllIndex::Sob}
            \big(s_0-s_1=\tfrac n{p_0}-\tfrac n{p_1}\ge0\text{ and }q_0\le q_1\big)\quad
            \text{or}\quad \big(s_0-s_1>\tfrac n{p_0}-\tfrac n{p_1}\ge0\big).
        \end{gather}
        Note that both cases are non-compact.

        \item\label{Item::EmbedAllIndex::Dom} Let $\Omega\subset\R^n$ be a bounded Lipschitz domain. Then $\Bs_{p_0q_0}^{s_0}(\Omega)\subseteq \Bs_{p_1q_1}^{s_1}(\Omega)$ holds if and only if
        \begin{gather}\label{Eqn::EmbedAllIndex::Dom1}
            s_0-s_1=\max(0,\tfrac n{p_0}-\tfrac n{p_1})\text{ and }q_0\le q_1;
            \\\label{Eqn::EmbedAllIndex::Dom2}
            \text{or}\quad s_0-s_1>\max(0,\tfrac n{p_0}-\tfrac n{p_1}).
        \end{gather}
        Moreover the embedding is non-compact if and only if \eqref{Eqn::EmbedAllIndex::Dom1} holds.
        
    \end{enumerate}
\end{cor}
\begin{proof}
    The proof is standard, see e.g. \cite[Chapter~2.7.1]{TriebelTheoryOfFunctionSpacesI}, \cite[Chapter~1.11.1]{TriebelTheoryOfFunctionSpacesIII} or \cite[Remark~2.87]{TriebelTheoryOfFunctionSpacesIV}. 
    
    Indeed one can check that $\ell^{q_0}(2^{j(s_{0} -\frac n{p_0}+\frac n2)}\cdot\ell^{p_0})_{j=0}^\infty\subset\ell^{q_1}(2^{j(s_{1} -\frac n{p_1}+\frac n2)}\cdot\ell^{p_1})_{j=0}^\infty$ holds if and only if \eqref{Eqn::EmbedAllIndex::Sob} holds; and the map is always non-compact. By Proposition~\ref{Prop::Wavelet}~\ref{Item::Wavelet::Isom} we get \ref{Item::EmbedAllIndex::Sob}.
    
    Similarly $\ell^{q_0}(2^{j(s_{0} -\frac n{p_0}+\frac n2)}\cdot\ell^{p_0}\{1,\dots,2^{nj}\})_{j=0}^\infty\subset\ell^{q_1}(2^{j(s_{1} -\frac n{p_1}+\frac n2)}\cdot\ell^{p_1}\{1,\dots,2^{nj}\})_{j=0}^\infty$ if and only if either \eqref{Eqn::EmbedAllIndex::Dom1} or \eqref{Eqn::EmbedAllIndex::Dom2} holds; and the map is compact if and only if \eqref{Eqn::EmbedAllIndex::Dom2} holds. 
    
    By Proposition~\ref{Prop::SubspaceWavelet}~\ref{Item::SubspaceWavelet::Small}  that $\Bs_{p_0q_0}^{s_0}(\Omega)\subset \Bs_{p_1q_1}^{s_1}(\Omega)$ implies \eqref{Eqn::EmbedAllIndex::Dom1} or \eqref{Eqn::EmbedAllIndex::Dom2}. By Proposition~\ref{Prop::SubspaceWavelet}~\ref{Item::SubspaceWavelet::Big} and using $\Lambda\circ E$ we see that \eqref{Eqn::EmbedAllIndex::Dom1} or \eqref{Eqn::EmbedAllIndex::Dom2} implies $\Bs_{p_0q_0}^{s_0}(\Omega)\subset \Bs_{p_1q_1}^{s_1}(\Omega)$. The compactness of the embedding $\Bs_{p_0q_0}^{s_0}(\Omega)\hookrightarrow \Bs_{p_1q_1}^{s_1}(\Omega)$ follows the same argument via Proposition~\ref{Prop::SubspaceWavelet}~\ref{Item::SubspaceWavelet::Small} and \ref{Item::SubspaceWavelet::Big}.
\end{proof}

\section{Embeddings Between Sequence Spaces}\label{Section::SeqSpace}

From the wavelet decomposition in the previous section, the problems are reduced to the embedding between sequence spaces. The goal of this section is to prove the following: let $0<q_0<q_1\le\infty$.
\begin{itemize}
    \item (Proposition~\ref{Prop::SeqEmbed1}) Let $p_0<p_1$. Then $\ell^{q_0}(\ell^{p_0})\hookrightarrow\ell^{q_1}(\ell^{p_1})$ is finitely strictly singular.
    \item (Corollary~\ref{Cor::DiagCorFSS}) Assume $p_1<\infty$. Then $\ell^{q_0}(L^\infty[0,1])\hookrightarrow\ell^{q_1}(L^{p_1}[0,1])$ is  finitely strictly singular.
    \item (Propositions~\ref{Prop::SeqEmbed2} and \ref{Prop::SeqEmbed2NotFSS}) Assume either $p_1 \leq p_0<\infty$ or $p_0=p_1=\infty$.  Then, 
    \begin{displaymath}
        \ell^{q_0}(2^{-j/p_0}\cdot\ell^{p_0}\{1,\dots,2^j\})_{j=0}^\infty\hookrightarrow\ell^{q_1}(2^{-j/p_1}\cdot\ell^{p_1}\{1,\dots,2^j\})_{j=0}^\infty
    \end{displaymath} is strictly singular but not finitely strictly singular.
\end{itemize}
Note that when $q_0=q_1$ all the above embeddings are not strictly singular.

These results are inspired by Plichko \cite{PlichkoStrictlySingular} and Lef\`evre, Rodr\'iguez-Piazza \cite{LefevrePiazzaFSSApplication}. Although some of the techniques might not be new, to the best of the authors' knowledge, the proof is not found in the standard literature. Thus we attach a full proof in detail.

\begin{lem}\label{Lem::TrivialNSS}
    Let $(T_j:X_j\to Y_j)_{j=0}^\infty$ be a family of bounded linear maps between quasi-normed spaces with uniformly bounded quasi-triangle constants such that $0<\limsup_{j\to\infty}\|T_j\|_{X_j\to Y_j}<\infty$. Define $\widehat T:(x_j)_{j=0}^\infty\mapsto(T_jx_j)_{j=0}^\infty$. 
    
    Then for every $0<q\le\infty$, $\widehat T:\ell^q(X_j)_{j=0}^\infty\to\ell^q(Y_j)_{j=0}^\infty$ is not strictly singular.
\end{lem}
\begin{proof}
    Let $e_j=(0,\dots,0,\underset j1,0,\dots)\in\ell^{\infty}$ be the standard elements. By assumption, there exist $c>0$, an infinite set $J\subset\N_0$, and $x_j\in X_j$ for each $j\in J$ such that $\|x_j\|_{X_j}=1$ and $\|T_jx_j\|_{Y_j}\geq c$.
    
    Next, consider the subspace $V := \Span\{ x_j \otimes e_j : j\in J\} \subset\ell^q(X_j)_{j=0}^\infty $. For any sequence $(a_{j})_{j\in J}\subset\R$, we have $\| \sum_{j\in J} a_{j} x_{j} \otimes e_j \|_{\ell^q(X_j)_j}=\|(a_j)_j\|_{\ell^q(J)}$ and $\| \widehat T\sum_{j\in J} a_{j} x_{j} \otimes e_{j}\|_{\ell^q(Y_j)_j}\ge c\|(a_j)_j\|_{\ell^q(J)}$, giving the non-strict singularity.
\end{proof}

In order to prove the required propositions, we need to prove some results on the ambient spaces and then apply them to their subspaces.

In the proof we apply the diagonal theorem from \cite[Theorem~3.3]{LefevrePiazzaFSSApplication}:
\begin{lem}\label{Lem::DiagThm}
    Let $(X_j,Y_j)_{j=0}^\infty$ be Banach spaces. Let $T_j:X_j\to Y_j$ ($j\ge0$) be bounded linear maps where the operator norms are uniformly bounded over $j\ge0$, i.e. $\sup_{j\ge0}\sup_{\|x_j\|_{X_j}=1}\|T_jx_j\|_{Y_j}<\infty$. 
      For given $0<p<q\le\infty$ we define $\widehat T_{p,q} : \ell^{p} (X_{j})_{j = 0}^{\infty} \to \ell^{q} (Y_{j})_{j = 0}^{\infty}$ by $\widehat T_{p,q}:(x_j)_{j=0}^\infty\mapsto(T_jx_j)_{j=0}^\infty$.
    
    Then $\widehat T_{p,q} $ is finitely strictly singular if  $(T_j)_j$ are uniformly finitely strictly singular. 
    
    More precisely,  if $\lim_{n\to\infty}\sup_{j\ge0}b_n(T_j)=0$, then we have  $\lim_{n\to\infty}b_n(\widehat T_{p,q})=0$. 
\end{lem}
\begin{rmk}
    When $1\le p<q\le\infty$, this is exactly the result in \cite{LefevrePiazzaFSSApplication}. It is possible that the diagonal theorem holds for quasi-Banach spaces $(X_j,Y_j)_{j=0}^\infty$ (with uniform quasi triangle inequalities) as well. But for our application this is already enough.
\end{rmk}
\begin{proof}[Proof of Lemma~\ref{Lem::DiagThm}] We prove the result by assuming \cite[Theorem~3.3]{LefevrePiazzaFSSApplication}. 

Using composition $\ell^p(X_j)_{j=0}^\infty\hookrightarrow \ell^{\max(p,1)}(X_j)_{j=0}^\infty\xrightarrow{\widehat T_{\max(p,1),\infty}}\ell^\infty(Y_j)_{j=0}^\infty$ and applying \cite[Theorem~3.3]{LefevrePiazzaFSSApplication} on $\widehat T_{\max(p,1),\infty}$  we see that $\widehat T_{p,\infty}$ is finitely strictly singular for all $0<p<\infty$.

When $q<\infty$, without loss of generality we can assume $\|T_j\|_{X_j\to Y_j}\le1$ for all $j\ge0$. Therefore
\begin{align*}
    \|\widehat T_{p,q}x\|_{\ell^q(Y_\bullet)}^q=&\sum_{j=0}^\infty\|T_jx_j\|_{Y_j}^q\le \sum_{j=0}^\infty\|x_j\|_{X_j}^p\|T_jx_j\|_{Y_j}^{q-p}\le\|x\|_{\ell^p(X_\bullet)}^p\sup_{j\ge0}\|T_jx_j\|_{Y_j}^{q-p}=\|x\|_{\ell^p(X_\bullet)}^p\|\widehat T_{p,\infty}x\|_{\ell^\infty(Y_\bullet)}^{q-p}.
\end{align*}

By assumption, for every $\eta > 0$ and every $n$-dim subspace $V\subset\ell^p(X_j)_{j=0}^\infty$, there is a unit vector $x\in V$ such that $\| \widehat{T}_{p,\infty}x\|_{\ell^\infty(Y_j)_j}\le (1+\eta)b_n(\widehat T_{p,\infty})$. Therefore $$\|\widehat T_{p,q}x\|_{\ell^q(Y_j)_j}\le \|\widehat T_{p,\infty}x\|_{\ell^\infty(Y_j)_j}^{1-\frac pq}\le\big((1+\eta) b_n(\widehat T_{p,\infty})\big)^{1-\frac pq}.$$ Since $\widehat T_{p,\infty}$ is finitely strictly singular by the above and $\eta > 0$ is arbitrary, the right hand side tends to $0$, which concludes the proof.
\end{proof}
\begin{cor}\label{Cor::DiagCorFSS}
    For every $0<p<\infty$ and $0<q_0<q_1\le\infty$, $\ell^{q_0}(L^\infty[0,1])\hookrightarrow \ell^{q_1}(L^p[0,1])$ is finitely strictly singular.
\end{cor}
\begin{proof}
    It is known that $L^\infty[0,1]\hookrightarrow L^p[0,1]$ is finitely strictly singular. See e.g. the proof of \cite[Theorem~5.2]{GrandpaRudin}. Although the statement \cite[Theorem~5.2]{GrandpaRudin} merely says that $L^\infty[0,1]\hookrightarrow L^p[0,1]$ is strictly singular, its proof gives the control $b_n(L^\infty[0,1]\hookrightarrow L^p[0,1])\le n^{-1/\max(p,2)}$, which tends to zero as $n\to\infty$. See also \cite[Corollary~29]{HernandezSemenovFSS}.

    Applying Lemma~\ref{Lem::DiagThm} the map $\ell^{q_0}(L^\infty[0,1])\hookrightarrow \ell^{q_1}(L^{\max(p,2)}[0,1])$ is also finitely strictly singular. Therefore by composing the map $\ell^{q_1}(L^{\max(p,2)}[0,1])\hookrightarrow \ell^{q_1}(L^p[0,1])$ we conclude the proof.
\end{proof}

The proof of Theorem~\ref{Thm::ClassifyRn}~\ref{Item::ClassifyRn::FSS} is reduced to the following statement on the sequence spaces.
\begin{prop}\label{Prop::SeqEmbed1}
    Let $0<p_0<p_1\le\infty$ and $0<q_0<q_1\le\infty$. 
    \begin{equation*}
         b_n\big(\ell^{q_0}(\ell^{p_0})\hookrightarrow\ell^{q_1}(\ell^{p_1})\big)\le n^{-\min(\frac1{p_0},\frac1{q_0})(1-\max(\frac{q_0}{q_1},\frac{p_0}{p_1}))}.
    \end{equation*}

    In particular $\ell^{q_0}(\ell^{p_0})\hookrightarrow\ell^{q_1}(\ell^{p_1})$ is finitely strictly singular.
\end{prop}

\begin{proof}
    The H\"older's inequality yields $\|x\|_{\ell^{q/\theta}(\ell^{p/\theta})}\le\|x\|_{\ell^q(\ell^p)}^\theta\|x\|_{\ell^\infty}^{1-\theta}$ for all $0<p,q\le\infty$ and $0<\theta<1$. Therefore,
    \begin{equation}\label{Eqn::SeqEmbed1::TmpInterpIneqn}
        \|x\|_{\ell^{q_1}(\ell^{p_1})}\le \|x\|_{\ell^{q_0\cdot\min(\frac{q_1}{q_0},\frac{p_1}{p_0})}(\ell^{p_0\cdot\min(\frac{q_1}{q_0},\frac{p_1}{p_0})})}\le\|x\|_{\ell^{q_0}(\ell^{p_0})}^{\max(\frac{q_0}{q_1},\frac{p_0}{p_1})}\|x\|_{c_0}^{1-\max(\frac{q_0}{q_1},\frac{p_0}{p_1})}.
    \end{equation}

    Recall from e.g. \cite[Lemma~4]{PlichkoStrictlySingular}, that if $V\subset c_0(\mathbb N_0)$ is a linear subspace of dimension $\ge n$, then there is a $x=(x_j)_{j=0}^\infty\in V$ such that $\max_{j\ge0}|x_j|$ is attained in at least $n$ points.
    
    Let $V\subset\ell^{q_0}(\ell^{p_0})$ be a given $n$ dimensional subspace. Therefore, there is a nonzero $x=(x_{j,k})_{j,k=0}^\infty\in V$ such that $\max_{j,k\ge0}|x_{j,k}|$ is attained in at least $n$ places. Suppose $\|x\|_{\ell^{q_0}(\ell^{p_0})}=1$, then we have
    \begin{equation*}
        1=\|x\|_{\ell^{q_0}(\ell^{p_0})}\ge\|x\|_{\ell^{\max(p_0,q_0)}}\ge (n\|x\|_{c_0}^{\max(p_0,q_0)})^{\max(p_0,q_0)^{-1}}\quad\Rightarrow\quad \|x\|_{c_0}\le n^{-\max(p_0,q_0)^{-1}}.
    \end{equation*}
    
    Using \eqref{Eqn::SeqEmbed1::TmpInterpIneqn} we conclude that
    \begin{equation*}
        b_n\big(\ell^{q_0}(\ell^{p_0})\hookrightarrow\ell^{q_1}(\ell^{p_1})\big)\le n^{-\max(p_0,q_0)^{-1}(1-\max(\frac{q_0}{q_1},\frac{p_0}{p_1}))}.\qedhere
    \end{equation*}
\end{proof}
\begin{cor}\label{Cor::SeqEmbed1}
    Let $s\ge t$, $0<p_0<p_1\le\infty$ and $0<q_0,q_1\le\infty$ such that either $s>t$ or $q_0<q_1$. Then the following embedding is finitely strictly singular, in particular strictly singular:
    $$\ell^{q_0}(2^{js}\cdot\ell^{p_0})_{j=0}^\infty\hookrightarrow\ell^{q_1}(2^{jt}\cdot\ell^{p_1})_{j=0}^\infty.$$
\end{cor}
\begin{proof}
    First we consider the case $q_0<q_1$ and  $s=t$. 
    Notice that the map $J_s(x_j)_{j=0}^\infty:=(2^{js}x_j)_{j=0}^\infty$ defines an isomorphism $J_s:\ell^q(2^{js}\cdot\ell^p)_{j=0}^\infty\to\ell^q(\ell^p)$ for every $0<p,q\le\infty$. Therefore $\ell^{q_0}(2^{js}\cdot\ell^{p_0})_{j=0}^\infty\hookrightarrow\ell^{q_1}(2^{js}\cdot\ell^{p_1})_{j=0}^\infty$ and $\ell^{q_0}(\ell^{p_0})\hookrightarrow\ell^{q_1}(\ell^{p_1})$ share the same Bernstein numbers. 
    The finite strict singularity follows immediately from Proposition~\ref{Prop::SeqEmbed1}.
    
    When $s>t$, we can decompose $\ell^{q_0}(2^{js}\cdot\ell^{p_0})_{j=0}^\infty\hookrightarrow\ell^{q_1}(2^{jt}\cdot\ell^{p_1})_{j=0}^\infty$ into
    \begin{equation*}
        \ell^{q_0}(2^{js}\cdot\ell^{p_0})_{j=0}^\infty \hookrightarrow 
        \ell^{\min(\frac{q_1}2,1)}(2^{jt}\cdot\ell^{p_0})_{j=0}^\infty\hookrightarrow\ell^{q_1}(2^{jt}\cdot\ell^{p_1})_{j=0}^\infty.
    \end{equation*}
    Here $\ell^{q_0}(2^{js}\cdot\ell^{p_0})_{j=0}^\infty\hookrightarrow\ell^{\min(q_1/2,1)}(2^{jt}\cdot\ell^{p_0})_{j=0}^\infty$ is bounded because 
    \begin{align*}
        \|x\|_{\ell^{\min(q_1/2,1)}(2^{jt}\cdot\ell^{p_0})_{j=0}^\infty}^{\min(q_1/2,1)}=&\sum_{j=0}^\infty(2^{jt}\|x_j\|_{\ell^{p_0}})^{\min(q_1/2,1)}\le\sup_{j\ge0}(2^{js}\|x_j\|_{\ell^{p_0}})^{\min(q_1/2,1)}\sum_{k=0}^\infty 2^{-k(s-t)\cdot \min(q_1/2,1)}
        \\\lesssim&_{s,t,q_1}\|x\|_{\ell^\infty(2^{js}\cdot\ell^{p_0})_{j=0}^\infty}^{\min(q_1/2,1)}\le \|x\|_{\ell^{q_0}(2^{js}\cdot\ell^{p_0})_{j=0}^\infty}^{\min(q_1/2,1)}.
    \end{align*}
    
    The result then follows immediately since $\ell^{\min(\frac{q_1}2,1)}(2^{jt}\cdot\ell^{p_0})_{j=0}^\infty\hookrightarrow\ell^{q_1}(2^{jt}\cdot\ell^{p_1})_{j=0}^\infty$ is already known to be finitely strictly singular.
\end{proof}

The key to prove Theorem~\ref{Thm::ClassifyDom}~\ref{Item::ClassifyDom::SS} is the following:
\begin{prop}\label{Prop::SeqEmbed2}
    Let $0<q_0<q_1\le\infty$. Let $X$ be a fixed quasi-Banach space and let $X_j\subset X$, $j=0,1,2,\dots$, be a sequence of finite dimensional subspaces with the norm induced from $X$.
    
    Then the following embedding is strictly singular:
    \begin{equation}\label{Eqn::SeqEmbed2SS}
        \ell^{q_0}(X_j)_{j=0}^\infty\hookrightarrow\ell^{q_1}(X_j)_{j=0}^\infty.
    \end{equation}
\end{prop}
\begin{rmk}\label{Rmk::SeqEmbed2::NotFSS}
    When $\sup_j\dim X_j=\infty$ \eqref{Eqn::SeqEmbed2SS} is not finitely strictly singular. Indeed, for any $n\ge1$, take a $j_n$ such that $\dim X_{j_n}\ge n$. Then the embedding \eqref{Eqn::SeqEmbed2SS} restricted to $X_{j_n}\otimes\{e_{j_n}\}$ is just a norm isometry, so $b_n\big(\ell^{q_0}(X_j)_{j=0}^\infty\hookrightarrow\ell^{q_1}(X_j)_{j=0}^\infty\big)\ge1$ for all $n\ge1$.

    For the special case $X_j=\ell^p\{1,\dots,2^j\}$, see Proposition~\ref{Prop::SeqEmbed2NotFSS} below for a stronger result.
\end{rmk} 
\begin{proof}[Proof of Proposition~\ref{Prop::SeqEmbed2}]
    By the Aoki--Rolewicz theorem (see e.g. \cite[Theorem~1.3, p.~7]{KaltonPeckRobertsFSpace}) there is a $0<p\le 1$ such that $X$ has a $p$-norm with equivalent norm topology: by re-norming $X$ we can assume that $\|a+b\|_X^p\le\|a\|_X^p+\|b\|_X^p$ for all $a,b\in X$.


    Let $\eps>0$ and $V\subset\ell^{q_0}(X_j)_{j=0}^\infty$ be an infinite dimensional subspace. We need to find a non-zero $x=(x_j)_{j=0}^\infty\in V$ such that $\|x\|_{\ell^{q_1}(X)}<\eps \|x\|_{\ell^{q_0}(X)}$.

    Let $N>1$ be such that $(N+1)^{1/q_1}N^{-1/q_0}<\frac\eps2$. Let $\delta=\delta_N\in(0,1)$ be such that 
    \begin{equation}\label{Eqn::SeqEmbed2::Assump}
        \frac{(1+(N+1)\delta^{\frac{\min(p,q_1)}{q_0}})^{\frac1{\min(p,q_1)}}}{(1-(N+1)\delta^{\frac{\min(p,q_0)}{q_0}})^{\frac1{\min(p,q_0)}}}<2.
    \end{equation}
    Such a $\delta$ always exists, since if we let $\delta\to0$ then the left hand side tends to $1$ which is less than $2$.
    
    Since $q_0<\infty$ and $X_j$ are all finite dimensional, we can find out a sequence $(j_n)_{n=1}^N\subset\N_0$ and a $(x^{(n)})_{n=1}^N\subset V$ with $\|x^{(n)}\|_{\ell^{q_0}(X)}=1$ for each $n\ge1$, such that
    \begin{itemize}
        \item $0\le j_1<j_2<\dots<j_N$;
        \item $x^{(n)}_j\equiv0$ for every $0\le j\le j_{n-1}$ (here we set $j_0=-1$);
        \item $\sum_{j=j_n+1}^\infty\|x^{(n)}_j\|_{X_j}^{q_0}\le\delta$, or equivalently, $\sum_{j=j_{n-1}+1}^{j_n}\|x^{(n)}_j\|_{X_j}^{q_0}\ge1-\delta$.
    \end{itemize}

    The construction is done by induction. We start with a unit vector $x^{(1)}\in V$, since $q_0<\infty$ there is a $j_1\ge0$ such that $\sum_{j=j_1+1}^\infty\|x^{(1)}_j\|_{X_j}^{q_0}\le\delta$.
    
    Suppose for $n\ge2$ and we have obtained $x^{(1)},\dots,x^{(n-1)}$ and $j_1<\dots<j_{n-1}$. Consider the projection $\Pc_{n-1}:\ell^{q_0}(\N_0;X)\twoheadrightarrow \ell^{q_0}(\{0,\dots,j_{n-1}\};X)$, we have $\operatorname{ker}(\Pc_{n-1}|_V)\neq\{0\}$ since $\dim \Pc_{n-1}(V)\le\sum_{j=0}^{j_{n-1}}\dim X_j<\infty$ but $\dim V=\infty$. Therefore, there must be a unit vector $x^{(n)}\in V$ such that $\Pc_{n-1}x^{(n)}=0$, i.e. $x^{(n)}_j\equiv0$ for every $0\le j\le j_{n-1}$. Take a $j_n>j_{n-1}$ be such that $\sum_{j=j_n+1}^\infty\|x^{(n)}_j\|_{X_j}^{q_0}\le\delta$ we close the induction.

    Now we see that $x:=x^{(1)}+\dots+x^{(N)}\in V$ is nonzero since
    \[
        \|x\|_{\ell^{q_0}(X)}\ge\|x\|_{\ell^{q_0}(\{0,\dots,j_1\};X)}
        =\|x^{(1)}\|_{\ell^{q_0}(\{0,\dots,j_1\};X)}
        \ge(1-\delta)^{1/q_0}>0.
    \]
    We are going to prove that $x$ is what we needed.

    Note that $\|u+v\|_{\ell^q(X)}^{\min(p,q)}
        \le \|u\|_{\ell^q(X)}^{\min(p,q)}
        +\|v\|_{\ell^q(X)}^{\min(p,q)}$ since $p\le1$. Therefore for each $1\le n\le N$,
    \begin{align*}
        \Big\|\sum_{k=1}^Nx^{(k)}\Big\|_{\ell^{q_0}(\{j_{n-1}+1,\dots,j_n\};X)}^{\min(p,q_0)}
        &=\Big\|\sum_{k=1}^nx^{(k)}\Big\|_{\ell^{q_0}(\{j_{n-1}+1,\dots,j_n\};X)}^{\min(p,q_0)}
        \\
        &\ge\Big\|x^{(n)}\Big\|_{\ell^{q_0}(\{j_{n-1}+1,\dots,j_n\};X)}^{\min(p,q_0)}
        -\Big\|\sum_{k=1}^{n-1}x^{(k)}\Big\|_{\ell^{q_0}(\{j_{n-1}+1,\dots,j_n\};X)}^{\min(p,q_0)}
        \\
        &\ge\Big\|x^{(n)}\Big\|_{\ell^{q_0}(\{j_{n-1}+1,\dots,j_n\};X)}^{\min(p,q_0)}
        -\sum_{k=1}^{n-1}\|x^{(k)}\|_{\ell^{q_0}(\{j_k+1,j_k+2,\dots\};X)}^{\min(p,q_0)}
        \\
        &\ge(1-\delta)^{\frac{\min(p,q_0)}{q_0}}
        -N\delta^{\frac{\min(p,q_0)}{q_0}}
        \ge1-(N+1)\delta^{\frac{\min(p,q_0)}{q_0}}.
    \end{align*}

    Taking the sum over $1\le n\le N$, we see that
    \begin{align*}
        \Big\|\sum_{k=1}^Nx^{(k)}\Big\|_{\ell^{q_0}(X)}\ge\bigg(\sum_{n=1}^N
        \Big\|\sum_{k=1}^Nx^{(k)}\Big\|_{\ell^{q_0}(\{j_{n-1}+1,\dots,j_n\};X)}^{q_0}\bigg)^{1/q_0}\ge N^\frac1{q_0}\Big(1-(N+1)\delta^{\frac{\min(p,q_0)}{q_0}}\Big)^{\frac1{\min(p,q_0)}}.
    \end{align*}

    On the other hand, using the property $\|x\|_{\ell^{q_1}(S;X)}\le\|x\|_{\ell^{q_0}(S;X)}$ for every $x\in\ell^{q_0}(X)$ and set $S$, we have, for each $1\le n\le N$,
    \begin{align*}
        \Big\|\sum_{k=1}^Nx^{(k)}\Big\|_{\ell^{q_1}(\{j_{n-1}+1,\dots,j_n\};X)}^{\min(p,q_1)}
        &=\Big\|\sum_{k=1}^nx^{(k)}\Big\|_{\ell^{q_1}(\{j_{n-1}+1,\dots,j_n\};X)}^{\min(p,q_1)}
        \\
        &\le\Big\|x^{(n)}\Big\|_{\ell^{q_1}(\{j_{n-1}+1,\dots,j_n\};X)}^{\min(p,q_1)}
        +\Big\|\sum_{k=1}^{n-1}x^{(k)}\Big\|_{\ell^{q_1}(\{j_{n-1}+1,\dots,j_n\};X)}^{\min(p,q_1)}
        \\
        &\le\Big\|x^{(n)}\Big\|_{\ell^{q_0}(\{j_{n-1}+1,\dots,j_n\};X)}^{\min(p,q_1)}
        +\sum_{k=1}^{n-1}\|x^{(k)}\|_{\ell^{q_1}(\{j_k+1,j_k+2,\dots\};X)}^{\min(p,q_1)}
        \\
        &\le1+N\delta^{\frac{\min(p,q_1)}{q_0}}.
    \end{align*}
    Similarly,
    \begin{align*}
        \Big\|\sum_{k=1}^Nx^{(k)}\Big\|_{\ell^{q_1}(\{j_N+1,j_N+2,\dots\};X)}^{\min(p,q_1)}
        &\le\sum_{k=1}^N
        \|x^{(k)}\|_{\ell^{q_1}(\{j_N+1,j_N+2,\dots\};X)}^{\min(p,q_1)}\le N\delta^{\frac{\min(p,q_1)}{q_0}}\le1+ N\delta^{\frac{\min(p,q_1)}{q_0}}.
    \end{align*}
    
    Therefore, using the convention $j_{N+1}=\infty$,
    \begin{align*}
        \Big\|\sum_{k=1}^Nx^{(k)}\Big\|_{\ell^{q_1}(X)}
        &=\Big(\sum_{n=1}^{N+1}
        \Big\|\sum_{k=1}^Nx^{(k)}\Big\|_{\ell^{q_1}(\{j_{n-1}+1,\dots,j_n\};X)}^{q_1}\Big)^\frac1{q_1}\le (N+1)^\frac1{q_1}
        \Big(1+(N+1)\delta^{\frac{\min(p,q_1)}{q_0}}\Big)^{\frac1{\min(p,q_1)}}.
    \end{align*}
    Plugging in \eqref{Eqn::SeqEmbed2::Assump} we conclude that
    \begin{equation*}
        \frac{\|x^{(1)}+\dots+x^{(N)}\|_{\ell^{q_1}(X)}}{\|x^{(1)}+\dots+x^{(N)}\|_{\ell^{q_0}(X)}}\le \frac{(N+1)^{1/q_1}}{N^{1/q_0}}\frac{(1+(N+1)\delta^{\frac{\min(p,q_1)}{q_0}})^{\frac{1}{\min(p,q_1)}}}{(1-(N+1)\delta^{\frac{\min(p,q_0)}{q_0}})^{\frac{1}{\min(p,q_0)}}}<\frac\eps 2\cdot2=\eps.
    \end{equation*}

    Therefore $x=x^{(1)}+\dots+x^{(N)}\in V$ is what we needed to achieve $\|x\|_{\ell^{q_1}(X)}<\eps \|x\|_{\ell^{q_0}(X)}$.
\end{proof}

The following statement is the non finitely strictly singular counterpart to Proposition~\ref{Prop::SeqEmbed2} (see also Remark~\ref{Rmk::SeqEmbed2::NotFSS}). For convenience, we only prove the case $m_j=2^j$.

\begin{prop}\label{Prop::SeqEmbed2NotFSS}
    Let $0<p_1\le p_0\le\infty$ and $0<q_0<q_1\le\infty$ such that $(p_0,p_1)\notin\{\infty\}\times(0,\infty)$. 
    
    Then for every $A\ge0$, the embedding is not finitely strictly singular:
    \begin{equation} \label{Eqn::SeqEmbed2NotFSS} 
    \ell^{q_0}(2^{-\frac j{p_0}}\cdot\ell^{p_0}\{1,\dots,2^j\})_{j=A}^\infty\hookrightarrow\ell^{q_1}(2^{-\frac j{p_1}}\cdot\ell^{p_1}\{1,\dots,2^j\})_{j=A}^\infty.
    \end{equation}
\end{prop}
\begin{proof}
    When $p_0=p_1$, we can take the $2^n$ dimensional subspace $V_n:=2^{-n/{p_0}}\cdot\ell^{p_0}\{1,\dots,2^n\}\hookrightarrow\ell^{q_0}(2^{-j/{p_0}}\cdot\ell^{p_0}\{1,\dots,2^j\})_{j=0}^\infty$, where $n\ge1$. The embedding \eqref{Eqn::SeqEmbed2NotFSS} restricted to $V_n$ is the identity map, hence the Bernstein numbers are all bounded from below by $1$, which does not tend to $0$.

    In the following we consider $0<p_1<p_0<\infty$. For each $n\ge1$ let $S_n:=(\{-1,1\}^n,\mu_n)$ be the probability space where $\mu_n\{a\}=2^{-n}$ for all $a\in\{-1,1\}^n$. For $1\le j\le n$, let $r_{n,j}:S_n\to\{-1,1\}$ be the standard projection on the $j$-th coordinate, i.e. $r_{n,j}(\epsilon_1,\dots,\epsilon_n)=\epsilon_j$. The standard Khintchine inequality yields $\|\sum_{j=1}^na_jr_{n,j}\|_{L^p(S_n,\mu_n)}\approx_p\|a\|_{\ell^2\{1,\dots,n\}}$ where the implied constant depends only on $p\in(0,\infty)$ but not on $n$. In particular, there is a $C_{p_0p_1}>0$ such that for every $n\ge1$ and $(a_j)_{j=1}^n\subset\R$,
    \begin{equation}\label{Eqn::KHIneqn}
        \Big\|\sum_{j=1}^na_jr_{n,j}\Big\|_{L^{p_1}(S_n,\mu_n)}\le \Big\|\sum_{j=1}^na_jr_{n,j}\Big\|_{L^{p_0}(S_n,\mu_n)}\le C_{p_0p_1}\Big\|\sum_{j=1}^na_jr_{n,j}\Big\|_{L^{p_1}(S_n,\mu_n)}.
    \end{equation}
    
    Since $(S_n,\mu_n)$ and $(\{1,\dots,2^n\},2^{-n}\cdot\#)$ are both probability spaces with uniform distributions on their sample sets, we see that they are isomorphic by taking any bijection $\Gamma_n :\{1,\dots,2^n\}\to S_n$ between sets. Thus we have elements $r_{n,j}\circ\Gamma_n\in\ell^{p_1}\{1,\dots,2^n\}$.
    
    Now take $X=\Span(r_{n,j}\circ\Gamma_n)_{j=1}^n$, we have $\dim X=n$. By \eqref{Eqn::KHIneqn} we get, for every $(a_j)_{j=1}^n\subset\R$,
    \begin{equation*}
        \Big\|\sum_{j=1}^na_j\cdot(r_{n,j}\circ\Gamma_n)\Big\|_{2^{-\frac n{p_0}}\cdot\ell^{p_0}}\le C_{p_0p_1}\Big\|\sum_{j=1}^na_j\cdot(r_{n,j}\circ\Gamma_n)\Big\|_{2^{-\frac n{p_1}}\cdot\ell^{p_1}}.
    \end{equation*}

    By considering the $n$-dimensional subspace $X \otimes e_{n} \subset \ell^{q_{0}}\big( 2^{- \frac{j}{p_{0}}} \cdot \ell^{p_{0}}\{ 1, \cdots, 2^{j} \} \big)_{j = A}^{\infty}$ where $n \geq A$, we conclude that $b_n\big(\ell^{q_0}(2^{-\frac j{p_0}}\cdot\ell^{p_0}\{1,\dots,2^j\})_{j=A}^\infty\hookrightarrow\ell^{q_1}(2^{-\frac j{p_1}}\cdot\ell^{p_1}\{1,\dots,2^j\})_{j=A}^\infty\big)\ge C_{p_0p_1}^{-1}$, which does not go to zero as $n\to\infty$.
\end{proof}

\begin{cor}\label{Cor::SeqEmbed2}
    Let $n,A\ge1$, $s\in\R$, $0<p_1\le p_0\le\infty$ and $0<q_0\le q_1\le\infty$. Consider the embedding \begin{equation}\label{Eqn::SeqEmbed2::Eqn}
        \ell^{q_0}(2^{j(s-n/p_0)}\cdot\ell^{p_0}\{1,\dots,2^{nj}\})_{j=A}^\infty\hookrightarrow\ell^{q_1}(2^{j(s-n/p_1)}\cdot\ell^{p_1}\{1,\dots,2^{nj}\})_{j=A}^\infty.
    \end{equation}
    \begin{enumerate}[(i)]
        \item\label{Item::SeqEmbed2::FSS} Suppose $q_0<q_1$ and $p_1<p_0=\infty$, then the embedding is finitely strictly singular.
        \item\label{Item::SeqEmbed2::SS} Suppose $q_0<q_1$, and either $p_1\le p_0<\infty$ or $p_0=p_1=\infty$, then the embedding is strictly singular but not finitely strictly singular.
        \item\label{Item::SeqEmbed2::NSS} Suppose $q_0=q_1$, then the embedding is not strictly singular.
    \end{enumerate}
\end{cor}
\begin{proof}
    $J_s(x_j)_{j=A}^\infty:=(2^{js}x_j)_{j=A}^\infty$ defines $J_s:\ell^q(2^{j(s-\frac np)}\cdot\ell^p\{1,\dots,2^{nj}\})_{j=A}^\infty\to\ell^q(2^{-j\frac np}\cdot\ell^p\{1,\dots,2^{nj}\})_{j=A}^\infty$ for every $0<p,q\le\infty$, which are all linear isomorphisms. So it suffices to consider the case $s=0$.

\medskip
    \noindent\ref{Item::SeqEmbed2::FSS}: For $j\ge0$ and $f:\{1,\dots,2^{nj}\}\to\R$ we define $\Gamma_j f:[0,1]\to\R$ by $\Gamma_j f(x):=f(\lceil 2^{nj}x\rceil)$ for $0<x\le1$ and $\Gamma_jf(0):=0$. Clearly $\Gamma_j:2^{-jn/p}\cdot\ell^p\{1,\dots,2^{nj}\}\to L^p[0,1]$ are isometric embeddings for all $0<p\le\infty$ and $j\ge0$. Therefore $(\Gamma_j)_{j=A}^\infty:\ell^q(2^{-jn/p}\cdot\ell^p\{1,\dots,2^{nj}\})_{j=A}^\infty\to \ell^q(L^p[0,1])$ is an isometry for all $0<p,q\le\infty$ and $A\ge1$.

    Applying Corollary~\ref{Cor::DiagCorFSS}, $\ell^{q_0}(\ell^\infty\{1,\dots,2^{nj}\})_{j=A}^\infty\to \ell^{q_1}(L^{p_1}[0,1])_{j=A}^\infty$ is finitely strictly singular. Pulling back by the topological embedding $(\Gamma_j)_{j=A}^\infty:\ell^{q_k}(2^{-j\frac n{p_k}}\cdot\ell^{p_k}\{1,\dots,2^{nj}\})_{j=A}^\infty\to \ell^{q_k}(L^{p_k}[0,1])$ for $k=0,1$ (see also Lemma~\ref{Lem::SubspaceArgument}) we conclude that \eqref{Eqn::SeqEmbed2::Eqn} is finitely strictly singular.

\medskip
    \noindent\ref{Item::SeqEmbed2::SS}: The non finitely strictly singular part follows from Proposition~\ref{Prop::SeqEmbed2NotFSS}. The strictly singular part follows  from applying Proposition~\ref{Prop::SeqEmbed2} to the second arrow below:
\begin{multline*}
    \ell^{q_0}(2^{j(s-\frac n{p_0})}\cdot\ell^{p_0}\{1,\dots,2^{nj}\})_{j=A}^\infty\xrightarrow{J_{s-n/p_0}}\ell^{q_0}(\ell^{p_0}\{1,\dots,2^{nj}\})_{j=A}^\infty\xhookrightarrow{\eqref{Eqn::SeqEmbed2SS}}\ell^{q_1}(\ell^{p_0}\{1,\dots,2^{nj}\})_{j=A}^\infty
    \\
    \xrightarrow{J_{n/{p_0}-s}}\ell^{q_1}(2^{j(s-\frac n{p_0})}\cdot\ell^{p_0}\{1,\dots,2^{nj}\})_{j=A}^\infty\hookrightarrow\ell^{q_1}(2^{j(s-\frac n{p_1})}\cdot\ell^{p_1}\{1,\dots,2^{nj}\})_{j=A}^\infty.
\end{multline*}
    \medskip
    \noindent\ref{Item::SeqEmbed2::NSS}: When $q_0=q_1$, we consider the infinite dimensional subspace $X=\Span(\1_{\{1,\dots,2^{nj}\}}\otimes e_j:j\ge A)$. Since $\|\1_{\{1,\dots,2^{nj}\}}\|_{2^{-jn/p}\cdot\ell^p}\equiv1$ for all $j\ge0$ and $0<p\le\infty$, the embedding restricted to $X$ is an isometry. Therefore, it has a bounded inverse, which proves the non-strict singularity. 
\end{proof}

\section{Proof of Theorems \ref{Thm::ClassifyDom} and \ref{Thm::ClassifyRn}}\label{Section::PfThm}

In this section we complete the proof of Theorems~\ref{Thm::ClassifyDom} and \ref{Thm::ClassifyRn}.

\begin{lem}\label{Lem::SubspaceArgument}
    Let $X_0,X_1,Y_0,Y_1$ be quasi-Banach spaces, and $\iota_0:X_0\to Y_0$ and $\iota_1:X_1\to Y_1$ be topological embeddings, i.e. $\iota_0(X_0)\subseteq Y_0$ and $\iota_1(X_1)\subseteq Y_1$ are both closed subspaces. Let $T_X:X_0\to X_1$ and $T_Y:Y_0\to Y_1$ be bounded linear maps such that $\iota_1\circ T_X=T_Y\circ\iota_0$.
    \begin{enumerate}[(i)]
        \item\label{Item::SubspaceArgument::Yes} Suppose $T_Y$ is (finitely) strictly singular. Then so is $T_X$.
        \item\label{Item::SubspaceArgument::No} Conversely suppose $T_X$ is not (finitely) strictly singular. Then neither is $T_Y$.
    \end{enumerate}
\end{lem}

Here in the lemma we have the commutative diagram \begin{tikzcd}
        Y_0\arrow{r}{T_Y}&Y_1\\X_0\arrow{r}{T_X}\arrow[u,hook,"\iota_0"]&X_1\arrow[u,hook,"\iota_1"]
    \end{tikzcd}.
\begin{proof}
    Since $\iota_0,\iota_1$ are both topological embeddings, there is a $C>0$ such that for $k=0,1$ and  for all $x\in X_k$, $C^{-1}\|x\|_{X_k}\le\|\iota_kx\|_{Y_k}\le C\|x\|_{X_k}$. Both results on the finite strict singularity follow from the fact that $b_n(T_X)\le C^2b_n(T_Y)$. The strict singularity follows from the same argument and we leave the details to the reader.
\end{proof}

We can decompose the proof of Theorems~\ref{Thm::ClassifyDom} and \ref{Thm::ClassifyRn} into a few critical cases.

For embeddings between Besov spaces with different levels of smoothness, we have:

\begin{prop}\label{Prop::SobEmbed}
    Let $0<p_0,p_1,q_0,q_1\le\infty$ and $s_0,s_1\in\R$ satisfy $s_0-s_1\ge\frac n{p_0}-\frac n{p_1}>0$, and $\Omega\subset\R^n$ be a bounded Lipschitz domain.
    \begin{enumerate}[(i)]
        \item\label{Item::SobEmbed::FSS} If $s_0-s_1>\frac n{p_0}-\frac n{p_1}$ or $q_0<q_1$, then $\Bs_{p_0q_0}^{s_0}(\R^n)\hookrightarrow \Bs_{p_1q_1}^{s_1}(\R^n)$ is finitely strictly singular. In particular $\Bs_{p_0q_0}^{s_0}(\Omega)\hookrightarrow \Bs_{p_1q_1}^{s_1}(\Omega)$ is also finitely strictly singular.
        \item\label{Item::SobEmbed::NSS} If $s_0-s_1=\frac n{p_0}-\frac n{p_1}$ and $q_0=q_1$, then $\Bs_{p_0q_0}^{s_0}(\Omega)\hookrightarrow \Bs_{p_1q_1}^{s_1}(\Omega)$ is not strictly singular. In particular $\Bs_{p_0q_0}^{s_0}(\R^n)\hookrightarrow \Bs_{p_1q_1}^{s_1}(\R^n)$ is also not strictly singular.
    \end{enumerate}
\end{prop}

\begin{rmk}
    When $s_0-s_1>\frac n{p_0}-\frac n{p_1}$, the embedding $\Bs_{p_0q_0}^{s_0}(\Omega)\to \Bs_{p_1q_1}^{s_1}(\Omega)$ is indeed compact, in particular it is finitely strictly singular. The only interesting case on the domain $\Omega$ is the case where $s_0-s_1=\frac n{p_0}-\frac n{p_1}$ and $q_0<q_1$.
\end{rmk}

\begin{proof}[Proof of Proposition~\ref{Prop::SobEmbed}]
    Take an $\eps>0$ where $p_0,q_0,p_1,q_1>\eps$ and $|s_0|,|s_1|<\eps^{-1}$. By Proposition~\ref{Prop::Wavelet}~\ref{Item::Wavelet::Isom}, $\Lambda:\Bs_{p_kq_k}^{s_k}(\R^n)\to \ell^{q_k}(2^{j(s_k-\frac n{p_k}+\frac n2)}\cdot\ell^{p_k}(\Z))_{j=0}^\infty$ are isomorphisms for $k=0,1$.

    \medskip\noindent\ref{Item::SobEmbed::FSS}:    
    By Corollary~\ref{Cor::SeqEmbed1} the embedding $\ell^{q_0}(2^{j(s_0-\frac n{p_0}+\frac n2)}\cdot\ell^{p_0}(\Z))_{j=0}^\infty\hookrightarrow\ell^{q_1}(2^{j(s_1-\frac n{p_1}+\frac n2)}\cdot\ell^{p_1}(\Z))_{j=0}^\infty$ is finitely strictly singular. Therefore $\Bs_{p_0q_0}^{s_0}(\R^n)\hookrightarrow \Bs_{p_1q_1}^{s_1}(\R^n)$ is also finitely strictly singular. Applying Lemma~\ref{Lem::SubspaceArgument}  with $X_k=\Bs_{p_kq_k}^{s_k}(\Omega)$, $Y_k=\Bs_{p_kq_k}^{s_k}(\R^n)$ and $\iota_k=E|_{X_k}$ (see Lemma~\ref{Lem::Extension}) for $k=0,1$, we get that $\Bs_{p_0q_0}^{s_0}(\Omega)\to \Bs_{p_1q_1}^{s_1}(\Omega)$ is also finitely strictly singular.

    \medskip\noindent\ref{Item::SobEmbed::NSS}: When $s_0-s_1=\frac n{p_0}-\frac n{p_1}$ and $q_0=q_1$, we are going to find an infinite dimensional subspace $X$ such that the embedding restricted to $X$ has bounded inverse. 
    
    Let $A\ge1$ be as in Proposition~\ref{Prop::SubspaceWavelet}. Then, there exists $m\in Q_A$ such that $\supp\psi_{jm}^g\subset\Omega$ for all $j\ge A$ and $g\in G_j(=\{0,1\}^n\backslash\{0\}^n)$. We fix a $g\in \{0,1\}^n\backslash\{0\}^n$ and consider $X:=\Span\{\psi_{jm}^g:j\ge A\}$. We see that $X$ is infinite dimensional and is a subspace of both $\Bs_{p_0q_0}^{s_0}(\Omega)$ and $\Bs_{p_0q_0}^{s_0}(\R^n)$. By Proposition~\ref{Prop::Wavelet}~\ref{Item::Wavelet::Isom} we have, for every $(a_j)_{j=A}^\infty \subset\R$,
    \begin{equation*}
        \Big\|\sum_{j=A}^\infty a_j\psi_{jm}^g\Big\|_{\Bs_{p_0q_0}^{s_0}(\R^n)}\approx\big\|\big(2^{j(s_0-\frac n{p_0}+\frac n2)}a_j\big)_{j=A}^\infty\big\|_{\ell^{q_0}}=\big\|\big(2^{j(s_1-\frac n{p_1}+\frac n2)}a_j\big)_{j=A}^\infty\big\|_{\ell^{q_1}}\approx\Big\|\sum_{j=A}^\infty a_j\psi_{jm}^g\Big\|_{\Bs_{p_1q_1}^{s_1}(\R^n)}.
    \end{equation*}
    By Lemma~\ref{Lem::TrivialInclusion} (and a possible change of the implied constants), the same estimate holds if we replace the norm $\|\cdot\|_{\Bs_{p_kq_k}^{s_k}(\R^n)}$ by $\|\cdot\|_{\Bs_{p_kq_k}^{s_k}(\Omega)}$ ($k=0,1$). This completes the proof.   
\end{proof}

For embeddings between Besov spaces with the same level of smoothness, we have:
\begin{prop}\label{Prop::DomEmbed}
    Let $s\in\R$, $0<p_1\le p_0\le\infty$ and $0<q_0\le q_1\le \infty$. Let $\Omega\subset\R^n$ be a bounded Lipschitz domain. Consider the embedding $\Bs_{p_0q_0}^{s}(\Omega)\hookrightarrow \Bs_{p_1q_1}^{s}(\Omega)$.
    \begin{enumerate}[(i)]
        \item\label{Item::DomEmbed::FSS} Suppose  $q_0<q_1$ and $p_1<p_0=\infty$. Then the embedding is finitely strictly singular.
        \item\label{Item::DomEmbed::SS} Suppose $q_0<q_1$, and either $p_0<\infty$ or $p_0=p_1=\infty$. Then the embedding is strictly singular but not finitely strictly singular.
        \item\label{Item::DomEmbed::NSS} Suppose $q_0=q_1$. Then the embedding is not strictly singular.
    \end{enumerate}
\end{prop}
\begin{proof}
    Still we take an $\eps>0$ such that $p_0,q_0,p_1,q_1>\eps$ and $|s|<\eps^{-1}$. Let $A\ge1$ and $R_j,S_j\subset Q_j\times G_j$ for $j\ge0$ be from Proposition~\ref{Prop::SubspaceWavelet} where we choose fixed subsets $V\Subset\Omega\Subset U$. Recall that $\# R_j=2^{(j-A)n}$ for $j\ge A$, $\#S_j=2^{(j+A)n}$ for $j\ge0$, and for $k=0,1$ both arrows in the compositions $\ell^{q_k}(2^{j(s-\frac n{p_k}+\frac n2)}\cdot\ell^{p_k}(R_j))_{j=A}^\infty\xrightarrow{\Lambda^{-1}} \Bs_{p_kq_k}^{s}(\Omega)\xrightarrow{\Lambda\circ E}\ell^{q_k}(2^{j(s-\frac n{p_k}+\frac n2)}\cdot\ell^{p_k}(S_j))_{j=0}^\infty$ are bounded.

    Denote $J:\Bs_{p_0q_0}^{s}(\Omega)\hookrightarrow \Bs_{p_1q_1}^{s}(\Omega)$, $I_R:\ell^{q_0}(2^{j(s-\frac n{p_0}+\frac n2)}\cdot\ell^{p_0}(R_j))_{j=A}^\infty\to \ell^{q_1}(2^{j(s-\frac n{p_1}+\frac n2)}\cdot\ell^{p_1}(R_j))_{j=A}^\infty$ and $I_S:\ell^{q_0}(2^{j(s-\frac n{p_0}+\frac n2)}\cdot\ell^{p_0}(S_j))_{j=0}^\infty\to \ell^{q_1}(2^{j(s-\frac n{p_1}+\frac n2)}\cdot\ell^{p_1}(S_j))_{j=0}^\infty$ for convenience. Recall from Corollary~\ref{Cor::SeqEmbed2} we see that for \ref{Item::SeqEmbed2::FSS}, \ref{Item::SeqEmbed2::SS} and \ref{Item::SeqEmbed2::NSS} respectively: 
    \begin{itemize}
        \item Suppose  $q_0<q_1$ and $p_1<p_0=\infty$. Then $I_S$ is finitely strictly singular.
        \item Suppose  $q_0<q_1$, and either $p_0<\infty$ or $p_0=p_1=\infty$. Then $I_S$ is strictly singular and $I_R$ is not finitely strictly singular.
        \item Suppose $q_0=q_1$. Then $I_R$ is not strictly singular.
    \end{itemize}
    The results then follow from applying Lemma~\ref{Lem::SubspaceArgument} with $(T_X,T_Y,\iota)=(I_R,J,\Lambda^{-1})$ and $(J,I_S,\Lambda\circ E)$.
\end{proof}

\begin{proof}[Proof of Theorem~\ref{Thm::ClassifyDom}]
    \ref{Item::ClassifyDom::Not} and \ref{Item::ClassifyDom::Cpt} follow from Corollary~\ref{Cor::EmbedAllIndex}~\ref{Item::EmbedAllIndex::Dom}. 
    
    \ref{Item::ClassifyDom::FSS}: By Corollary~\ref{Cor::EmbedAllIndex}~\ref{Item::EmbedAllIndex::Dom} the map is noncompact. For finite strict singularity, when $s_0-s_1=\frac n{p_0}-\frac n{p_1}$ the result follows from Proposition~\ref{Prop::SobEmbed}~\ref{Item::SobEmbed::FSS}; when $s_0-s_1=0$ the result follows from Proposition~\ref{Prop::DomEmbed}~\ref{Item::DomEmbed::FSS}.
    
    \ref{Item::ClassifyDom::SS}: The result follows from Proposition~\ref{Prop::DomEmbed}~\ref{Item::DomEmbed::SS}.

    \ref{Item::ClassifyDom::NSS}: When $s_0-s_1=\frac n{p_0}-\frac n{p_1}$ the result follows from Proposition~\ref{Prop::SobEmbed}~\ref{Item::SobEmbed::NSS}; when $s_0-s_1=0$ the result follows from Proposition~\ref{Prop::DomEmbed}~\ref{Item::DomEmbed::NSS}.
\end{proof}

\begin{proof}[Proof of Theorem~\ref{Thm::ClassifyRn}]
    \ref{Item::ClassifyRn::Not} follows from Corollary~\ref{Cor::EmbedAllIndex}~\ref{Item::EmbedAllIndex::Sob}. 

    \ref{Item::ClassifyRn::FSS}: By Corollary~\ref{Cor::EmbedAllIndex}~\ref{Item::EmbedAllIndex::Sob} the map is noncompact. The finite strict singularity follows from Proposition~\ref{Prop::SobEmbed}~\ref{Item::SobEmbed::FSS}.

    \ref{Item::ClassifyRn::NSS}: When $p_0<p_1$, i.e. the case $s_0-s_1=\frac n{p_0}-\frac n{p_1}>0$ and $q_0=q_1$, the result follows from Proposition~\ref{Prop::SobEmbed}~\ref{Item::SobEmbed::NSS}. The only remaining case left from the previous discussion is the case $p_0=p_1$. Recall that either $s_0-s_1>0$, or $s_0-s_1=0$ but $q_0\le q_1$.

    Indeed, choose an $\eps>0$ such that $p_0(=p_1),q_0,q_1>\eps$ and $|s_0|,|s_1|<\eps^{-1}$. Let $\psi_{jm}^g$ be the wavelet elements in Proposition~\ref{Prop::Wavelet}. We can consider the infinite dimensional subspace $X=\Span\{\psi_{0m}^0:m\in\Z^n\}$. By Proposition~\ref{Prop::Wavelet}~\ref{Item::Wavelet::Isom}, for every $(a_m)_{m\in\Z^n}\subset\R$,
    \begin{equation*}
        \Big\|\sum_{m\in\Z^n}a_m\psi_{0m}^0\Big\|_{\Bs_{p_0q_0}^{s_0}(\R^n)}\approx\|(a_m\big)_{m\in\Z^n}\|_{\ell^{p_0}}=\|(a_m\big)_{m\in\Z^n}\|_{\ell^{p_1}}\approx\Big\|\sum_{m\in\Z^n}a_m\psi_{0m}^0\Big\|_{\Bs_{p_1q_1}^{s_1}(\R^n)}.
    \end{equation*}
    Thus the embedding restricted to $X$ has bounded inverse, completing the whole proof.
\end{proof}

\section{On Homogeneous Besov Embeddings}\label{Section::HomoEmbed}

As a bonus result to Theorem~\ref{Thm::ClassifyRn}, we state and prove the analogous classification result on homogeneous Besov spaces by reviewing the methods used in the previous discussion.

Recall that the homogeneous Besov space $\dot \Bs_{pq}^s(\R^n)$ is defined to be the space of elements in the quotient space $\dot\Ss'(\R^n)=\Ss'(\R^n)/\{\text{polynomials}\}$ for which
\[
\|f\|_{\dot \Bs_{pq}^s(\R^n)} = \|f\|_{\dot \Bs_{pq}^s(\R^n;\varphi)} := \Big(\sum_{j\in\Z} \| 2^{js}\varphi_j\ast f\|_{L^p(\R^n)}^q\Big)^{1/q} <\infty.
\]
Here $(\varphi_j)_{j\in\Z}\subset\dot\Ss(\R^n)$ is a fixed family of Schwartz functions with all moments vanishing, such that $\supp\hat\varphi_0\subset\{\frac12<|\xi|<2\}$, $\hat\varphi_j(\xi)=\hat\varphi_0(2^{-j}\xi)$ for all $j\in\Z$, and $\sum_{j\in\Z}\hat\varphi_j(\xi)=1$ for all $\xi\neq0$. See e.g. \cite[Chapter~5]{TriebelTheoryOfFunctionSpacesI} for details.
\begin{thm}\label{Thm::ClassifyHomo}
    Let $p_0,p_1,q_0,q_1\in(0,\infty]$, $s_0,s_1\in\R$. Consider the embedding $\dot \Bs_{p_0q_0}^{s_0}(\R^n)\hookrightarrow\dot \Bs_{p_1q_1}^{s_1}(\R^n)$.
    \begin{enumerate}[(i)]
        \item\label{Item::ClassifyHomo::Not} $\dot \Bs_{p_0q_0}^{s_0}(\R^n)\not\subset \dot \Bs_{p_1q_1}^{s_1}(\R^n)$, i.e. there is no embedding, if and only if one of the following occurs:
        \begin{itemize}
            \item $s_0-s_1\neq\frac n{p_0}-\frac n{p_1}$;
            \item $p_0>p_1$ (i.e. $\frac n{p_0}-\frac n{p_1}<0$) or $q_0>q_1$.
        \end{itemize}
        \item\label{Item::ClassifyHomo::FSS} The embedding is non-compact but finitely strictly singular iff $s_0-s_1=\frac n{p_0}-\frac n{p_1}>0$ and $q_0<q_1$.
        \item\label{Item::ClassifyHomo::NSS} The embedding is not strictly singular if and only if one of the following occurs:
        \begin{itemize}
            \item $s_0-s_1=\frac n{p_0}-\frac n{p_1}=0$ and $q_0\le q_1$;
            \item $s_0-s_1=\frac n{p_0}-\frac n{p_1}\ge0$ and $q_0=q_1$.
        \end{itemize}
\end{enumerate}
\end{thm}
If we adapt the notation of index sets in Proposition~\ref{Prop::Wavelet}, the orthonormal basis in the proof below can be written as $\{\dot\psi_{jm}^g:j\in\Z,m\in 2^{-j}\Z^n, g\in \{0,1\}^n\backslash\{0\}^n\}\subset\dot\Ss(\R^n)$. It is more convenient to use  identifications $(2^{-j}\Z^n)\times(\{0,1\}^n\backslash\{0\}^n)\simeq\N_0$ on the index $(m,g)$, because we do not need the analogy of Proposition~\ref{Prop::Wavelet}~\ref{Item::Wavelet::Supp} in the proof.
\begin{proof}
    Firstly the wavelet decomposition holds for homogeneous function spaces, see e.g. \cite[Theorem~6.16]{FrazierJawerthWeissWavelet}: there is a collection $\{\dot\psi_{jv}:j\in\Z,v\in \N_0\}\subset\dot\Ss(\R^n)$ which forms an orthonormal basis of $L^2(\R^n)$, such that for every $f\in\dot\Ss'(\R^n)$ we have decomposition $f=\sum_{j\in\Z}\sum_{v=0}^\infty(f,\dot\psi_{jv})\dot\psi_{jv}$ with the series converging in $\dot\Ss'(\R^n)$. Moreover, $\Lambda (f):=\{(f,\dot\psi_{jv}):v\in\N_0\}_{j\in\Z}$ defines a map which gives the isomorphisms
    \begin{equation*}
        \Lambda:\dot\Bs_{pq}^s(\R^n)\xrightarrow{\simeq}\ell^q(2^{j(s-\frac np+\frac n2)}\cdot\ell^p)_{j=-\infty}^\infty,\qquad \text{for all }0<p,q\le\infty\text{ and }s\in\R.
    \end{equation*}
    So it is equivalent to consider the embedding $\ell^{q_0}(2^{j(s_0-\frac n{p_0}+\frac n2)}\cdot\ell^{p_0})_{j=-\infty}^\infty\hookrightarrow\ell^{q_1}(2^{j(s_1-\frac n{p_1}+\frac n2)}\cdot\ell^{p_1})_{j=-\infty}^\infty$.

\smallskip\noindent
    \ref{Item::ClassifyHomo::Not}: By assumption $\ell^{q_0}(2^{j(s_0-\frac n{p_0}+\frac n2)}\cdot\ell^{p_0})_{j=-\infty}^\infty\not\subset\ell^{q_1}(2^{j(s_1-\frac n{p_1}+\frac n2)}\cdot\ell^{p_1})_{j=-\infty}^\infty$, thus $\dot \Bs_{p_0q_0}^{s_0}(\R^n)\not\subset \dot \Bs_{p_1q_1}^{s_1}(\R^n)$. 
    
    To be more precise, when $s_0-s_1\neq\frac n{p_0}-\frac n{p_1}$ the result follows from that $2^{j(s_1-\frac n{p_1}+\frac n2)}\big/2^{j(s_0-\frac n{p_0}+\frac n2)}\to\infty$ for either $j\to+\infty$ or $j\to-\infty$. 
    
    When $s_0-\frac n{p_0}+\frac n2=s_1-\frac n{p_1}+\frac n2=:s$, the map $J_s(x_j)_{j\in\Z}:=(2^{js}x_j)_{j\in\Z}$ defines isomorphisms that for $i=0,1$, $J_s:\ell^{q_i}(2^{js}\cdot\ell^{p_i})_{j\in\Z}\to\ell^{q_i}(\Z;\ell^{p_i})$ . Since $\ell^{q_0}(\ell^{p_0})\not\subset\ell^{q_1}(\ell^{p_1})$ holds when $q_0>q_1$ or $p_0>p_1$, pulling back from $J_s$ we get the result.

\smallskip\noindent
    \ref{Item::ClassifyHomo::FSS}: In this assumption, applying Proposition~\ref{Prop::SeqEmbed1} the map $\ell^{q_0}(\Z;\ell^{p_0})\hookrightarrow\ell^{q_1}(\Z;\ell^{p_1})$ is finitely strictly singular. Clearly it is non-compact. Pulling back from $J_s$, we get the non-compactness and finite strict singularity of $\ell^{q_0}(2^{j(s_0-\frac n{p_0}+\frac n2)}\cdot\ell^{p_0})_{j=-\infty}^\infty\hookrightarrow\ell^{q_1}(2^{j(s_1-\frac n{p_1}+\frac n2)}\cdot\ell^{p_1})_{j=-\infty}^\infty$, so is $\dot \Bs_{p_0q_0}^{s_0}(\R^n)\hookrightarrow \dot \Bs_{p_1q_1}^{s_1}(\R^n)$.

\smallskip\noindent\ref{Item::ClassifyHomo::NSS}: When $q_0=q_1$, the result follows from Lemma~\ref{Lem::TrivialNSS}; when $p_0=p_1$ the result follows from the fact that for every $j\in\Z$, the spaces $2^{j(s_0-\frac n{p_0}+\frac n2)}\cdot\ell^{p_0}$ and $2^{j(s_1-\frac n{p_1}+\frac n2)}\cdot\ell^{p_1}$ coincide and have the same norms.
\end{proof}

\section{Further Remarks}\label{Section::Further}

We conclude our paper with a few remarks and open questions.

In Theorems~\ref{Thm::ClassifyDom} and \ref{Thm::ClassifyRn}, we observe the impact of differences between function spaces on bounded Lipschitz domains \(\Omega\) and on \(\mathbb{R}^n\). A natural question that arises is what results can be obtained for embeddings on unbounded open sets (say unbounded uniform domains).

We believe that the upper bounds in Proposition~\ref{Prop::SeqEmbed1} could be sharpened, and that it should be possible to derive corresponding lower bounds. This could lead to improved estimates on the rate of decay for Bernstein numbers between Besov spaces when the embedding is non-compact but finitely strictly singular. Note that although we do not have a canonical definition of Besov norms, the decay rates of their Bernstein numbers are determined up to a constant multiple which still makes sense. The case where \(\Omega = \mathbb{R}^n\) is particularly intriguing, which is equivalent to computing the decay rate of $b_n\big(\ell^{q_0}(\ell^{p_0})\to\ell^{q_1}(\ell^{p_1})\big)$.

Given the obtained results for Besov spaces, it seems natural to ask about the inner structure for non-compact embeddings between Triebel-Lizorkin spaces. This could be quite interesting, given their different behaviors of the corresponding sequence spaces from the wavelet decomposition, which makes the study of Triebel-Lizorkin spaces more challenging. Also, it is worth noting that the obtained results for Besov spaces could be used for non-compact Sobolev embeddings which will generalize and extend results from \cite{LangMihula} and \cite{BouGro}. 

\begin{ack}
    The authors would like to thank Luis Rodr\'iguez-Piazza for some helpful discussion.
\end{ack}

\subsection*{AI disclaimer} ChatGPT was used only to assist with proofreading and typo fixing during the final stages of revision  for publication. The structure of the paper and the arguments in all proofs were established in the first version and remained unchanged.
The author has checked the resulting arguments and takes full responsibility for the mathematical content of this paper.

\bibliographystyle{amsalpha} 
\bibliography{reference} 

\end{document}